\documentclass[reqno,12pt]{amsart}
\usepackage{color}
\usepackage{amsthm}
\usepackage{amsmath}
\usepackage{epsfig}
\usepackage{graphicx}
\usepackage{amssymb}
\usepackage{latexsym}
\usepackage{pdfpages}

\usepackage{enumitem}

\usepackage{tikz}

\usetikzlibrary{babel}
\usetikzlibrary{calc}

\usepackage{ulem} %para poder tachar: \sout{loquequierotachar}

\usepackage{ifpdf}
\newcommand{\res}{\!\!\mathop{\hbox{
                                \vrule height 7pt width .5pt depth 0pt
                                \vrule height .5pt width 6pt depth 0pt}}
                                \nolimits}

\ifpdf %if using pdfLaTeX in PDF mode
  \usepackage[hidelinks]{hyperref}
\else %if using LaTeX or pdfLaTeX in DVI mode
\fi %%% fin del ifpdf

 \usepackage[usenames,dvipsnames]{pstricks}
 \usepackage{epsfig}
 \usepackage{pst-grad} % For gradients
 \usepackage{pst-plot} % For axes

\usepackage{color}

\newtheorem{theorem}{Theorem}[section]
\newtheorem{lemma}[theorem]{Lemma}
\newtheorem{definition}[theorem]{Definition}
\newtheorem{proposition}[theorem]{Proposition}
\newtheorem{corollary}[theorem]{Corollary}
\newtheorem{remark}[theorem]{Remark}
\newtheorem{example}[theorem]{Example}
\newtheorem*{theorem*}{\it Theorem}

\def\vint_#1{\mathchoice%
          {\mathop{\kern 0.2em\vrule width 0.6em height 0.69678ex depth -0.58065ex
                  \kern -0.8em \intop}\nolimits_{\kern -0.4em#1}}%
          {\mathop{\kern 0.1em\vrule width 0.5em height 0.69678ex depth -0.60387ex
                  \kern -0.6em \intop}\nolimits_{#1}}%
          {\mathop{\kern 0.1em\vrule width 0.5em height 0.69678ex
              depth -0.60387ex
                  \kern -0.6em \intop}\nolimits_{#1}}%
          {\mathop{\kern 0.1em\vrule width 0.5em height 0.69678ex depth -0.60387ex
                  \kern -0.6em \intop}\nolimits_{#1}}}
\def\vintslides_#1{\mathchoice%
          {\mathop{\kern 0.1em\vrule width 0.5em height 0.697ex depth -0.581ex
                  \kern -0.6em \intop}\nolimits_{\kern -0.4em#1}}%
          {\mathop{\kern 0.1em\vrule width 0.3em height 0.697ex depth -0.604ex
                  \kern -0.4em \intop}\nolimits_{#1}}%
          {\mathop{\kern 0.1em\vrule width 0.3em height 0.697ex depth -0.604ex
                  \kern -0.4em \intop}\nolimits_{#1}}%
          {\mathop{\kern 0.1em\vrule width 0.3em height 0.697ex depth -0.604ex
                  \kern -0.4em \intop}\nolimits_{#1}}}

\def\R{\mathbb R}
\def\N{\mathbb N}

\numberwithin{equation}{section}

\def\1{\raisebox{2pt}{\rm{$\chi$}}}
\def\z{{\bf z}}
\def\w{{\bf w}}
\def\e{{\bf e}}

\def\v{{\bf v}}

\renewcommand{\v}{\mathrm{v}}
\newcommand{\vf}{\mathrm{f}}% para los nodos\u
\newcommand{\vi}{\mathrm{i}}% para los nodos\v
\newcommand{\V}{\mathrm{V}}% flia de nodos
\newcommand{\E}{\mathrm{E}}%flia de aristas

\numberwithin{equation}{section}

\def\1{\raisebox{2pt}{\rm{$\chi$}}}

\definecolor{violet(ryb)}{rgb}{0.53, 0.0, 0.69}
\definecolor{internationalorange}{rgb}{1.0, 0.31, 0.0}
\usepackage{mathtools}
\begin{document}

\title[The Total Variation Flow in Metric Graphs ]{\bf The Total Variation Flow in Metric Graphs }

\author[J. M. Maz\'on]{Jos\'e M. Maz\'on}

 \address{ J. M. Maz\'{o}n: Departamento de An\'{a}lisis Matem\'{a}tico,
Univ. Valencia, Dr. Moliner 50, 46100 Burjassot, Spain.
 {\tt mazon@uv.es}}

\keywords{Metric graphs, functions of total variation, total variation flow, the $1$-Laplacian, asymptotic
behaviour.\\
\indent 2010 {\it Mathematics Subject Classification: 05C21, 47J35, 35K67.}
}

\setcounter{tocdepth}{1}

\date{\today}

\begin{abstract}
Our aim is to study the Total Variation Flow in Metric Graphs. First, we define the functions of bounded
variation in Metric Graphs and  their total variation, we also give an integration by parts formula. We
prove existence and uniqueness of solutions and that the solutions
  reach the  mean of the initial
data in finite time.
 Moreover, we obtain explicit solutions.
\end{abstract}

\maketitle

{\it Dedicated to the memory  of Ireneo Peral.}

{ \renewcommand\contentsname{Contents}
\setcounter{tocdepth}{3}
\addtolength{\parskip}{-0.2cm}
{\small \tableofcontents}
\addtolength{\parskip}{0.2cm} }

\section{Introduction and Preliminaries}

Metric graphs  are widely used to model a wide range of problems in chemistry, physics, or engineering,
describing quasi-one-dimensional systems such as carbon nano-structures,
quantum wires, transport networks, or thin waveguides. Concerning  the applications in biology,
we can cite, for instance, the recent works \cite{DLPZ} and \cite{JPS}. They are also widely studied
in mathematics; see \cite{BK} and \cite{Mugnolo}  for an overview.

 One of the earliest accounts of a partial differential equation set on a metric graph can be found in the
 $1980$ work of Lumer (\cite{Lumer}) on ramification spaces. Since then, the theory
has seen considerable developments, due, in particular, to the natural appearance of graphs in the
modeling
of various physical situations. Among the partial differential equation problems set on metric graphs, one
has become increasingly popular: the ones set on quantum graphs.
By quantum graphs one usually refers to a metric graph $\Gamma  = (V, E)$ equipped with a
differential operator $H$ often referred to as the Hamiltonian. The most popular example of a Hamiltonian
is $- \Delta$ on the edges with Kirchhoff conditions. The book of Berkolaiko and Kuchment \cite{BK}
provides an
excellent introduction to the theory of quantum graphs. In the last years, we have had a great development
of other important topics like: the  wave equation in metric graphs related with control problems (see
survey book \cite{DZ}) and  nonlinear quantum graphs associated with  the nonlinear evolution equation of
Schr\"{o}dinger type (see  the survey paper \cite{Noja}). Now, to our knowledge, there is very little
literature on nonlinear evolution problems in metric graphs, such as for the $p$-Laplacian operator.

The aim of this paper is to analyse the initial-boundary value problem associated with the total variation
flow in  metric graphs.
In this regard, we introduce the $1$-Laplacian operator associated with a metric graph.  We then proceed
to prove existence and uniqueness of solutions of the total variation flow in metric graphs for data in
$L^2(\Gamma)$ and to study their asymptotic behaviour, showing that the solutions reach the stationary
state in finite time. Furthermore, we obtain explicit solutions.

From the mathematical point of view, the study of the total variation flow in Euclidean spaces  was
carried out using, as its main tools, the classical theory of maximal monotone operators due to Brezis
(\cite{Brezis}) and the Crandall-Liggett Theorem (\cite{CrandallLiggett}, \cite{Barbu}), being the energy
space the space of function of bounded variation.  In order to  characterize the solutions,  the Green's
formula shown by Anzelotti in \cite{Anzellotti}  proved to be crucial (see \cite{ABCM0}, \cite{ABCM} and
\cite{ACMBook} for a survey).  The study of  a similar problem in the general framework  of metric random
walk spaces, which have as important particular cases the weighted graphs and nonlocal problems with
non--singular kernels, was done in \cite{MST1}.

Here, we use similar tools, so we introduce the space of bounded variation  functions in metric graphs and
we establish a Green's formula in order to characterize the $1$-Laplacian operator in metric graphs. Let
me point out  the importance of giving an adequate definition of the total variation of a bounded
variation function in the context of metric graphs that takes into account its structure and measures the
jumps in the vertices.

\subsection{Metric  Graphs} We  recall here some basic knowledge about metric graphs, see for instance
\cite{BK} and the references therein.

A graph $\Gamma$ consists of a finite or countable infinite set
of vertices $\V(\Gamma)=\{\v_i\}$ and a set of edges $\E(\Gamma)=\{\e_j\}$ connecting the
vertices. A graph $\Gamma$ is said to be a finite graph if the number of edges
and the number of vertices are finite.
 An edge and a vertex on
that edge are called incident. We will denote $\v\in \e$ when the edge $\e$ and the vertex $\v$ are
incident.
We define $\E_{\v}(\Gamma)$ as the set of all edges incident to $\v$, and the  {\it degree} of $\v$ as
$d_\v:=  \sharp \E_{\v}(\Gamma)$. We define the {\it boundary} of $V(\Gamma)$ as
$$\partial V(\Gamma):= \{ \v \in V(\Gamma) \ : \ d_\v =1 \},$$
and its {\it interior} as
$${\rm int}( V(\Gamma)) := \{ \v \in V(\Gamma) \ : \ d_\v > 1 \}.$$

We will assume the absence of loops,
since if these are present, one can break them into pieces by introducing
new intermediate vertices. We also assume the absence of multiple edges.

A {\it walk} is a sequence of edges $\{\e_1,\e_2,\e_3,\dots\}$ in which, for each $i$ (except the last),
the end of  $\e_i$  is
the beginning of $\e_{i+1}$. A {\it trail} is a walk in which no edge is repeated.
A {\it path} is a trail in which no vertex is repeated.

From now on we will deal with a connected, compact and metric graph~$\Gamma$:

\noindent $\bullet$ A  graph $\Gamma$ is a metric graph if
\begin{enumerate}
	\item each edge $\e$ is assigned  with a positive length $\ell_{\e}\in]0,+\infty];$

\item   for each edge $\e$, a coordinate is assigned to each point of it, including its vertices. For
    that purpose, each   edge $\e$ is identified with an ordered pair
$(\vi_{\e},\vf_{\e})$ of vertices, being $\vi_{\e}$ and $\vf_{\e}$ the
initial and terminal
vertex of $\e$ respectively, which has no  sense of meaning when travelling along the path  but allows
us to define  coordinates by means of an increasing function
	$$
	\begin{array}{rlcc}
 c_\e:&\e&\to& [0,\ell_\e]\\
 &x&\rightsquigarrow& x_{\e}
 \end{array}
$$
such that, letting $c_\e(\vi_\e):=0$ and $c_\e(\vf_\e):=\ell_{\e}$, it is exhaustive;  $x_{\e}$ is
called the coordinate of the point $x\in \e$.
\end{enumerate}

\noindent $\bullet$ A graph is said to be
connected if a path exists between every pair of vertices, that is, a graph which is
connected in the usual topological sense.

\noindent  $\bullet$ A compact metric graph is a finite metric graph whose edges  all have finite
length.

\medskip

If a sequence of edges $\{\e_j\}_{j=1}^n$  forms a path, its length is
defined as $\sum_{j=1}^n\ell_{\e_j}.$  The length of a metric graph, denoted $\ell(\Gamma)$, is the sum of
the length of all its edges. Sometime we identify $\Gamma$ with $$\Gamma \equiv\bigcup_{\e \in E(\Gamma)}
\e.$$

Given a set $A \subset \Gamma$, we define its {\it length} as
$$\ell (A):= \sum_{ \e \in E(\Gamma), A \cap \e \not=\emptyset} \mathcal{L}^1( c_{\e}(A \cap \e)).$$

For two vertices $\v$ and $\hat\v,$ the distance between $\v$ and $\hat \v$, $d_\Gamma(\v,\hat \v)$, is
defined as the
minimal length of the paths connecting them.  Let us be more precise and consider $x$, $y$ two points in
the graph $\Gamma$.

-if $x,y\in\e$ (they belong to the same edge, note that they can be vertices), we define {\it the
distance-in-the-path-$\e$}   between $x$ and $y$  as $$\hbox{dist}_{\e}(x,y):=|y_\e-x_\e|;$$

-if $x\in \e_a$, $y\in \e_b$, let $P=\{\e_a,\e_1,\dots,\e_{n},\e_b\}$ be a path ($n\ge 0$) connecting
them.
 Let us call $\e_{0} = \e_a$ and $\e_{n+1}= \e_b$. Following the definition given above for a path, set
 $\v_{0}$ the vertex that is the end of  $\e_0$  and
the beginning of  $\e_{1}$ (note that these vertices need not be the  terminal and the initial vertices of
the edges that are taken into account), and   $\v_{n}$   the vertex that is the end of  $\e_n$ and the
beginning of $\e_{n+1}$.
We will say that   the {\it distance-in-the-path-$P$} between $x$ and $y$ is equal to
$$\hbox{dist}_{\e_0}(x,\v_0)+ \sum_{1\le j\le n}\ell_{\e_j}+ \hbox{dist}_{\e_{n+1}}(\v_n,y) .$$
We define the distance between $x$ and $y$, that we will denote by $d_\Gamma(x,y)$, as the infimum  of all
the {\it distances-in-paths} between $x$ and $y$, that is,
$$
\begin{array}{lr}d_\Gamma(x,y)&= \inf\Big\{ \hbox{dist}_{\e_0}(x,\v_0)+ \sum_{1\le j\le n}\ell_{\e_j}+
\hbox{dist}_{\e_{n+1}}(\v_n,y):\qquad\qquad \\[10pt] & \qquad\qquad \hbox{
$\{\e_0,\e_1,\dots,\e_{n},\e_{n+1}\}$   path connecting $x$ and $y$} \Big\}.
\end{array}
$$

 We remark that the distance between two points $x$ and $y$ belonging to the same edge $\e$ can be
 strictly smaller than $|y_\e-x_\e|$. This happens when there is a path connecting them (using  more edges
 than $\e$) with length smaller than $|y_\e-x_\e|$.

A function $u$ on a metric graph $\Gamma$ is a collection of functions $[u]_{\e}$
defined on
$]0,\ell_{\e}[$ for all $\e\in \E(\Gamma),$
not just at the vertices as in discrete models.

Throughout this work, $ \int_{\Gamma} u(x)  dx$ or  $ \int_{\Gamma} u$ denotes
$ \sum_{\e\in \E(\Gamma)} \int_{0}^{\ell_{\e}} [u]_{\e}(x_\e)\, dx_\e$. Note that given $\Omega \subset
\Gamma$, we have
$$\ell(\Omega) = \int_\Gamma \1_\Omega dx.$$

Let $1\le p\le +\infty.$ We say that $u$ belongs to  $L^p(\Gamma)$ if
$[u]_{\e}$ belongs to $L^p(]0,\ell_{\e}[)$ for all $\e\in \E(\Gamma)$ and
\[
	\|u\|_{L^{p} (\Gamma)}^p\coloneqq\sum_{\e\in \E(\Gamma)}
	\|[u]_{\e}\|_{L^{p}(0,\ell_{\e})}^p<+\infty.
\]
 The Sobolev space $W^{1,p}(\Gamma)$ is defined as
the space of  functions $u$ on $\Gamma$ such  that
$[u]_{\e}\in W^{1,p}(0,\ell_{\e})$
for all $\e\in \E(\Gamma)$ and
\[
	\|u\|_{W^{1,p}(\Gamma)}^p\coloneqq \sum_{\e\in \E(\Gamma)}
	\|[u]_{\e}\|_{L ^p(0,\ell_{\e})}^p+\|[u]_{\e}{}^\prime\|_{L ^p(0,\ell_{\e})}^p<+\infty.
\]
The space $W^{1,p}(\Gamma)$ is a Banach space for $1 \le p \le\infty$.
It is reflexive for $1 < p < \infty$ and separable for $1 \le p < \infty.$
Observe that in the definition of $W^{1,p}(\Gamma)$ we does not assume the continuity at the vertices.

A quantum graph is a metric graph $\Gamma$ equipped with a
differential operator
acting on the edges together with vertex conditions.
In this work, we will consider the $1-$Laplacian differential operator given by
\[
	\Delta_1 u(x):=
	\left(\frac{ u^{\prime}(x)}{|u^{\prime}(x)|}\right)^{\prime},
\]
on each edge.

\section{The total variation flow in metric graphs}

In this section we will assume that $\Gamma$ is a finite, compact and connected metric graph. To introduce
the total variation flow in the metric graph $\Gamma$, we first need to study the bounded variation
functions in $\Gamma$ and  to get a Green's formula in $\Gamma$ analogue to the classical Anzellotti
Green's formula.
\subsection{BV functions and integration by parts}

For bounded variation functions of one variable we follow \cite{AFP}.
Let $I \subset \R$  be an interval, we say that a
function $u \in L^1(I)$ is of bounded variation if its
distributional derivative $Du$ is a Radon measure on $I$ with
bounded total variation $\vert Du \vert (I) < + \infty$.
 We denote
by $BV(I)$ the space of all functions of bounded variation in $I$.
It is well known (see \cite{AFP}) that given $u \in BV(I)$ there
exists $\overline{u}$ in the equivalence class of $u$, called a
good representative of $u$, with the following properties. If $J_u$
is the set of atoms of $Du$, i.e., $x \in J_u$ if and only if $Du(\{
x \}) \not= 0$, then $\overline{u}$ is continuous in $I \setminus
J_u$ and has a jump discontinuity at any point of $J_u$:
$$\overline{u}(x_{-})  := \lim_{y \uparrow x}\overline{u}(y) = Du(]a,x[), \ \ \ \
\ \overline{u}(x_{+}) := \lim_{y \downarrow x}\overline{u}(y) =
Du(]a,x]) \ \ \ \forall \, x \in J_u,$$ where by simplicity we are
assuming that $I = ]a,b[$. Consequently,
$$\overline{u}(x_{+}) - \overline{u}(x_{-})
 = Du(\{ x \}) \ \ \ \forall \, x \in J_u.$$
Moreover, $\overline{u}$ is differentiable at ${\mathcal L}^1$
     a.e. point of $I$, and the derivative $\overline{u}'$ is the
density of $Du$ with respect to ${\mathcal L}^1$.
For $u \in BV(I)$, the  measure $Du$ decomposes into its absolutely continuous and
singular parts $Du = D^a u + D^s u$. Then $D^a u = \overline{u}' \ {\mathcal L}^1$. We
also split $D^su$ in two parts: the {\it jump} part $D^j u$ and the {\it Cantor} part
$D^c u$.

 It is well known (see for instance \cite{AFP}) that
$$D^j u = Du \res J_u = \sum_{x \in J_u} \overline{u}(x_{+}) - \overline{u}(x_{-}),$$
and also,
$$\vert Du \vert (I) = \vert D^au \vert (I) + \vert D^j u \vert (I) + \vert D^c u \vert
(I) $$ $$= \int_a^b \vert \overline{u}'(x) \vert \, dx + \sum_{x
\in J_u} \vert \overline{u}(x_{+}) - \overline{u}(x_{-}) \vert +
\vert D^c u \vert (I).$$
Obviously, if $u \in BV(I)$ then $u \in W^{1,1}(I)$ if and only if
$D^su \equiv 0$, and in this case we have $Du = \overline{u}' \
{\mathcal L}^1$.

A measurable subset $E \subset I$ is a set of {\it finite perimeter} in $I$ if $\1_E \in BV(I)$, and its
perimeter is defined as
$${\rm Per}(E, I):= \vert D \1_E \vert (I).$$

From now on, when we deal with point-wise valued  $BV$-functions we  shall always use the
good representative. Hence, in the case $u \in W^{1,1}(I)$, we shall assume that  $u \in
C(\overline{I})$.

Given $\z \in W^{1,2}(]a,b[)$ and $u \in BV(]a,b[)$, by $\z Du$ we mean the Radon measure
in $]a,b[$ defined as
$$\langle \varphi, \z Du \rangle := \int_a^b \varphi \z \, Du \ \
\ \ \ \ \forall \, \varphi \in C_c(]a,b[).$$

Note that if $\varphi \in \mathcal{D}(]a,b[):= C^\infty_c(]a,b[)$, then
$$\langle \varphi, \z Du \rangle = - \int_a^b u \z^{\prime} \varphi dx - \int_a^b u \z \varphi^{\prime}
dx,$$
which is  the definition given by Anzellotti in \cite{Anzellotti}.

Working as in \cite[Corollary 1.6]{Anzellotti}, it is easy to see that
\begin{equation}\label{boound}
\vert \z Du \vert (B) \leq \Vert \z \Vert_{L^{\infty}(]a,b[)} \vert Du \vert (B) \quad \hbox{for all
Borelian} \ B \subset ]a,b[.
\end{equation}

Then, $\z Du$ is absolutely continuous  with  respect to the measure $\vert Du \vert$, and we will denote
by $\theta(\z, Du, x)$ the Radom-Nikodym derivative of $\z Du$  with  respect to $\vert Du \vert$, that
is
 $$\int_a^b \z Du = \int_a^b \theta(\z, Du, x) d \vert Du \vert (x).$$ Working as in \cite[Proposition
 2.8]{Anzellotti}, we have that if $f \in C^1(\R)$ is an increasing function, then
\begin{equation}\label{compp}
\theta(\z, D(f(u)), x) = \theta(\z, Du, x) \quad \vert Du \vert-\hbox{a.e. in} \ ]a,b[.
\end{equation}

The next result was proved in \cite{Anzellotti} in $\R^N$, with $N \geq 2$. We can adapt the proof for
$N=1$. For convenience, we give here the details.

\begin{proposition}\label{megusta1} Let $\z_n \in W^{1,2}(]a,b[)$. If
$$\lim_{n \to \infty}\z_n = \z \quad \hbox{weakly$^*$ in} \ L^\infty (]a,b[),$$
and
$$\lim_{n \to \infty}\z^{\prime}_n = \z^{\prime} \quad \hbox{weakly in} \ L^1 (]a,b[),$$
then for every $u \in BV(]a,b[)$, we have
\begin{equation}\label{EE1}
\z_n Du \to \z Du \quad \hbox{as measures},
\end{equation}
and
\begin{equation}\label{EE2}
\lim_{n \to \infty} \int_a^b \z_n Du = \int_a^b\z Du.
\end{equation}
\end{proposition}
\begin{proof} We have
$$M:=\sup_{n \in \N} \Vert \z_n \Vert_\infty < \infty, \quad \hbox{and then} \quad \Vert \z \Vert_\infty
\leq M.$$
Then,
$$\left\vert \int_a^b \z_n Du \right\vert \leq M \int_a^b \vert Du \vert.$$
Thus, to verify that \eqref{EE1} holds; that is,
\begin{equation}\label{EE3}
\lim_{n \to \infty} \int_a^b  \varphi \z_n Du = \int_a^b  \varphi \z Du
\end{equation}
for every $\varphi \in C_c(]a,b[)$, it is sufficient to check this limit for  test functions $\varphi \in
\mathcal{D}(]a,b[)$. Now, for $\varphi \in \mathcal{D}(]a,b[)$,
$$ \int_a^b  \varphi \z_n Du  = - \int_a^b u \z_n^{\prime} \varphi dx - \int_a^b u \z_n \varphi^{\prime}
dx \to - \int_a^b u \z^{\prime} \varphi dx - \int_a^b u \z \varphi^{\prime} dx =  \int_a^b  \varphi \z
Du,$$
which proves \eqref{EE1}. Let us prove now \eqref{EE2}.  Given $\epsilon >0$, since $\vert Du \vert$ is a
bounded Radon measure, there exists an open subset $U \subset ]a,b[$ such that
\begin{equation}\label{EE4}
\int_{]a,b[ \setminus U} \vert Du \vert \leq \frac{\epsilon}{4M}
\end{equation}
and for every $\varphi \in \mathcal{D}(]a,b[)$, there exists $N \in \N$ such that
\begin{equation}\label{EE5}
\left\vert \int_a^b  \varphi \z_n Du - \int_a^b  \varphi \z Du \right\vert < \frac{\epsilon}{2}, \quad
\forall \ n \geq N.
\end{equation}
Now, we choose  $\varphi \in \mathcal{D}(]a,b[)$ such that $0 \leq \varphi \leq1$, $\varphi \equiv 1$ on
$\overline{U}$. Then, by \eqref{EE4} and \eqref{EE5}, for all  $n \geq N$, we have
$$\left\vert \int_a^b  \z_n Du - \int_a^b \z Du \right\vert  \leq \left\vert \int_a^b  \varphi \z_n Du -
\int_a^b  \varphi \z Du \right\vert + \int_a^b (1 - \varphi) \vert \z_n Du \vert  + \int_a^b (1 - \varphi)
\vert \z Du \vert  $$ $$ \leq \frac{\epsilon}{2} + \int_{]a,b[ \setminus U} \vert \z_n Du \vert +
\int_{]a,b[ \setminus U} \vert \z_n Du \vert \leq \frac{\epsilon}{2} +2 M \int_{]a,b[ \setminus U} \vert
Du \vert \leq \epsilon$$
proving  \eqref{EE2}.

\end{proof}

 We need the following
integration by parts formula, which can be proved using a suitable
regularization of $u \in BV(I)$ as in the proof of \cite[Theorem 1.9]{Anzellotti} (see also  Theorem C.9.
of
\cite{ACMBook}).

\begin{lemma}\label{IntBP} If $\z \in W^{1,2}(]a,b[)$ and $u \in
BV(]a,b[)$, then
\begin{equation}\label{EIntBP}
\int_a^b \z Du + \int_a^b u(x) \z^{\prime}(x) \, dx = \z(b) u(b_{-})- \z(a) u(a_{+}).
\end{equation}
\end{lemma}

\begin{definition}{\rm
We define the set of {\it bounded variation functions} in $\Gamma$ as
$$BV(\Gamma):= \{ u \in L^1(\Gamma) \ : \ [u]_{\e}\in BV(]0,\ell_{\e}[) \ \hbox{for all} \ \e\in \E(\Gamma)
\}.$$

Given $u \in BV(\Gamma)$, for $\e \in E_\v$, we define
$$[u]_\e(\v) := \left\{ \begin{array}{ll} [u]_\e(0+), \quad &\hbox{if} \ \ \v = \vi_{\e} \\[10pt]
[u]_\e(\ell_\e-), \quad &\hbox{if} \ \ \v = \vf_{\e}. \end{array}  \right.$$

For $u \in BV(\Gamma)$, we define
$$\vert D u \vert (\Gamma):=  \sum_{\e\in \E(\Gamma)} \vert D [u]_{\e}  \vert(]0,\ell_{\e}[).$$
We also write $$\vert D u \vert (\Gamma) =\int_{\Gamma} |Du|.$$

}
\end{definition}

Obviously, for $u \in BV(\Gamma)$, we have
\begin{equation}\label{const}
 \vert D u \vert (\Gamma)= 0 \ \iff \ [u]_\e \ \hbox{is constant in} \ ]0, \ell_\e[, \ \ \forall \, \e \in
 E(\Gamma).
\end{equation}

$BV(\Gamma)$ is a Banach space  with  respect to the norm

	$$\|u\|_{BV(\Gamma)}\coloneqq \Vert u \Vert_{L^1(\Gamma)} +  \vert D u \vert (\Gamma).$$

\begin{remark}{\rm
Note that we do not include a continuity condition at the vertices as in the definition of the spaces
$BV(\Gamma)$.  This is due to the fact that, if we include the continuity in the vertices, then typical
functions of bounded variation such as the functions of the form $\1_D$ with $D \subset \Gamma$ such that
$\v \in D$, being $\v$ a common vertex to two edges, would not be elements of $BV(\Gamma)$.
}
$\blacksquare$
\end{remark}

By the Embedding Theorem for $BV$-functions (cf. \cite[Corollary 3.49, Remark 3.30]{AFP}), we have the
following result.
\begin{theorem}\label{embedding} The embedding $BV(\Gamma) \hookrightarrow L^p(\Gamma)$ is continuous for
$1\leq p \leq \infty$, being compact for $1 \leq p < \infty$.
\end{theorem}

 We denote
 $$\mathcal{D}(\Gamma):= \bigoplus_{\e \in E(\Gamma)} C^\infty_c (]0,\ell_{\e}[),$$
 and
 $$C_c(\Gamma):= \bigoplus_{\e \in E(\Gamma)} C_c (]0,\ell_{\e}[).$$

 $C_c(\Gamma)$ is a Banach space  with  respect to the norm $\Vert u \Vert_{\infty}  = \sup \{ \vert u(x)
 \vert \ : \ x \in \Gamma \}$, we denote by
 $$\mathcal{M}_b(\Gamma) := \left( C_c(\Gamma)\right)^*,$$
 the dual of $C_c(\Gamma)$, and we will call the elements of $\mathcal{M}_b(\Gamma) $  {\it Radon
 measures} in $\Gamma$.

\begin{definition}{\rm
Given $u \in BV(\Gamma)$, we define $Du: C_c(\Gamma) \rightarrow \R$ as
$$\langle Du, \varphi \rangle:= \sum_{\e\in \E(\Gamma)} \int_0^{\ell_{\e}} \varphi_\e \, dD[u]_\e.$$
Note that if $\varphi \in \mathcal{D}(\Gamma)$, then
$$\langle Du, \varphi \rangle = - \sum_{\e\in \E(\Gamma)} \int_0^{\ell_{\e}} \varphi^{\prime}_\e \, [u]_\e
dx$$
}
\end{definition}

We have
$$\left\vert \langle Du, \varphi \rangle \right\vert \leq \sum_{\e\in \E(\Gamma)} \left\vert
\int_0^{\ell_{\e}} \varphi_\e \, dD[u]_\e \right\vert \leq \sum_{\e\in \E(\Gamma)} \Vert \varphi_\e
\Vert_\infty \left\vert D[u]_\e \right\vert (0, \ell_{\e}) = \Vert \varphi \Vert_\infty \vert D u \vert
(\Gamma).$$
Therefore, $Du \in \mathcal{M}_b(\Gamma)$ and $\Vert Du \Vert_{\mathcal{M}_b(\Gamma)} \leq \vert D u \vert
(\Gamma)$. On the other hand, given $\epsilon > 0$ there exists $\varphi_\e \in C_c ((0,\ell_{\e}))$, with
$\Vert \varphi_\e \Vert_\infty \leq 1$ such that $$\vert D[u]_\e \vert (0, \ell_{\e}) \leq \langle
D[u]_\e, \varphi_\e \rangle + \frac{\epsilon}{\vert E(\Gamma) \vert}.$$
Then, if $\varphi:= \bigoplus_{\e \in E(\Gamma)} \varphi_\e \in C_c(\Gamma)$, we have
$$\vert D u \vert (\Gamma) = \sum_{\e\in \E(\Gamma)} \vert D [u]_{\e}  \vert(0,\ell_{\e}) \leq \sum_{\e\in
\E(\Gamma)} \langle D[u]_\e, \varphi_\e \rangle + \epsilon = \langle Du, \varphi \rangle + \epsilon \leq
\Vert Du \Vert_{\mathcal{M}_b(\Gamma)}+ \epsilon.$$
Consequently,
\begin{equation}\label{equal2}
\vert D u \vert (\Gamma)= \Vert Du \Vert_{\mathcal{M}_b(\Gamma)} \quad \hbox{for all} \ \ u \in
BV(\Gamma).
\end{equation}

Let us point out that, in metric graphs, $\vert D u \vert (\Gamma)(u)$ is  not a good definition of the
total variation of $u$ since it does not measure the jumps of the function at the vertices. In order to
give a definition of the total variation of a function $ u \in BV(\Gamma)$ that takes into account the
jumps of the function  at the vertices, we are going to  obtain a Green's formula like the one  obtained
by Anzellotti in \cite{Anzellotti} for $BV$-functions in Euclidean spaces. In order to do this, we start
by defining the pairing $\z Du$ between  an element $\z \in W^{1,2}(\Gamma)$ and
a BV function $u$. This will be a metric graph analogue of the classic Anzellotti pairing
introduced in \cite{Anzellotti}.

\begin{definition}{\rm
For $\z \in W^{1,2}(\Gamma)$ and $u \in BV(\Gamma)$, we define $\z Du:= ( [\z]_\e, D[u_\e])_{\e \in
E(\Gamma)} $, that is, for $\varphi \in C_c(\Gamma)$,
$$\langle \z Du, \varphi \rangle = \sum_{\e\in \E(\Gamma)} \int_0^{\ell_{\e}} \varphi_\e[\z]_\e \,
D[u]_\e.$$
We have that $\z Du$ is a Radon measure in $\Gamma$
and
$$\int_\Gamma \z Du =  \sum_{\e\in \E(\Gamma)} \int_0^{\ell_{\e}} [\z]_\e \, D[u]_\e.$$

}
\end{definition}

By \eqref{boound}, we have

\begin{equation}\label{booundN}
\left\vert \int_{\Gamma} \z Du \right\vert \leq \Vert \z \Vert_{L^{\infty}(\Gamma)} \vert D u \vert
(\Gamma).
\end{equation}
If we define
$$\theta(\z, Du, x) := \sum_{\e\in \E(\Gamma)}\theta([\z]_\e, D[u]_\e, x),$$
then
$$\int_\Gamma \z Du = \int_\Gamma \theta(\z, Du, x) d \vert Du \vert (x).$$
Moreover, by \eqref{compp}, if $f \in C^1(\R)$ is a increasing function, then
\begin{equation}\label{Ncompp}
\theta(\z, D(f(u)), x) = \theta(\z, Du, x) \quad \vert Du \vert-\hbox{a.e. in} \ \Gamma.
\end{equation}

Given $\z \in W^{1,2}(\Gamma)$, for $ \e \in E_\v$, we define
$$[\z]_\e (\v):= \left\{ \begin{array}{ll}[\z]_\e(\ell_{\e}) \quad &\hbox{if} \ \ \v = \vf_\e, \\[10pt]
-[\z]_\e(0),\quad &\hbox{if} \ \ \v = \vi_\e.  \end{array} \right..$$

By Lemma \ref{IntBP}, we have
            $$\int_{\Gamma} \z Du  =  \sum_{\e\in \E(\Gamma)} \int_0^{\ell_{\e}}[\z]_\e \, D[u]_\e $$ $$=
            - \sum_{\e\in \E(\Gamma)} \int_0^{\ell_{\e}}[u]_\e(x) ([\z]_\e)^{\prime}(x) dx+ \sum_{\e\in
            \E(\Gamma)} ( [\z]_\e(\ell_{\e}) [u]_\e((\ell_{\e})_{-}) - [\z]_\e(0) [u]_\e (0_+) )$$ $$= -
            \int_\Gamma u\z^{\prime} + \sum_{\v \in V(\Gamma)} \sum_{\e\in \E_\v(\Gamma)} [\z]_\e(\v)
            [u]_\e(\v).$$

Then,
 if we define
$$\int_{\partial \Gamma} \z u:=\sum_{\v \in V(\Gamma)} \sum_{\e\in \E_\v(\Gamma)} [\z]_\e(\v)
[u]_\e(\v),$$
for $\z \in W^{1,2}(\Gamma)$ and $u \in BV(\Gamma)$,
we have  the following {\it Green's formula}:
\begin{equation}\label{intbpart}
\int_{\Gamma} \z Du + \int_\Gamma u\z^{\prime} = \int_{\partial \Gamma} \z u.
\end{equation}

 We define
 $$X_0(\Gamma):= \{ \z \in W^{1,2}(\Gamma) \  : \ \z(\v) =0, \ \ \forall \v \in  V(\Gamma)\}.$$

 For $u \in BV(\Gamma)$ and $\z \in X_0(\Gamma)$, we have the following {\it Green's formula}
 \begin{equation}\label{0intbpart}
\int_{\Gamma} \z Du + \int_\Gamma u\z^{\prime} = 0.
\end{equation}

 \begin{proposition}\label{TVreg} For $u \in BV(\Gamma)$, we have
\begin{equation}\label{Form100}
\vert D u \vert (\Gamma) =   \sup \left\{ \displaystyle\int_{\Gamma} u(x) \z^{\prime}(x) dx  \ : \ \z \in
X_0(\Gamma), \ \Vert \z \Vert_{L^\infty(\Gamma)} \leq 1 \right\}.
\end{equation}
 \end{proposition}
\begin{proof} Let $u \in BV(\Gamma)$. Given $\z \in X_0(\Gamma)$ with \ $\Vert \z \Vert_{L^\infty(\Gamma)}
\leq 1$, applying Green's formula \eqref{0intbpart} and \eqref{booundN}, we have
$$
  \int_\Gamma u\z^{\prime} = -\int_{\Gamma} \z Du \leq \vert D u \vert (\Gamma).$$
Therefore,
$$\sup \left\{ \int_{\Gamma} u(x) \z^{\prime}(x) d(x)  \ : \ \z \in X_0(\Gamma), \ \Vert \z
\Vert_{L^\infty(\Gamma)} \leq 1 \right\} \leq \vert D u \vert (\Gamma).   $$
On the other hand,
$$\vert D u \vert (\Gamma)= \sum_{\e\in \E(\Gamma)} \vert D [u]_{\e}  \vert(0,\ell_{\e}) = \sum_{\e\in
\E(\Gamma)} \sup \left\{\int_0^{\ell_{\e}} [u]_\e \varphi^{\prime}_\e  \ : \  \varphi_\e \in
C^\infty_c((0,\ell_{\e}) ), \ \Vert \varphi_\e \Vert_\infty \leq 1 \right\}.$$
Now, given $(\varphi_\e) \in \mathcal{D}(\Gamma)$, if we define $\z$ such that $[\z]_\e = \varphi_\e$ for
all $\e \in E(\Gamma)$, we have $\z \in X(\Gamma)$. Hence, we get
$$\vert D u \vert (\Gamma) \leq  \sup \left\{ \int_{\Gamma} u(x) \z^{\prime}(x) d(x)  \ : \ \z \in
X_0(\Gamma), \ \Vert \z \Vert_{L^\infty(\Gamma)} \leq 1 \right\}.$$
\end{proof}

\begin{remark}{\rm  By the above result,  we have that  the energy functional  $\mathcal{E}_\Gamma :
L^2(\Gamma) \rightarrow [0, + \infty]$ defined by
$$\mathcal{E}_\Gamma(u):= \left\{ \begin{array}{ll} \displaystyle
\vert D u \vert (\Gamma)
 \quad &\hbox{if} \ u\in  BV(\Gamma), \\[10pt] + \infty \quad &\hbox{if } u\in L^2(\Gamma)\setminus
 BV(\Gamma), \end{array} \right.$$
 is convex and lower semi-continuous.  Therefore, we could study the gradient flow associated with
 $\mathcal{E}_\Gamma$  as a possible definition of the total variation flow in metric graphs. However, I
 would like to point out that this is not the adequate way since the solutions of this gradient flow
 coincide with the solutions of the Neumann problem at each edge, regardless of the structure of the
 metric graph. This is the reason for which we are going to introduce our concept of total variation in
 metric graphs.  $\blacksquare$
}
\end{remark}

We consider now the elements of $W^{1,2}(\Gamma)$ that satisfies a  {\it Kirchhoff condition}, that is,
the set
 $$X_K(\Gamma):= \left\{ \z \in W^{1,2}(\Gamma) \  : \ \sum_{\e \in  E_\v(\Gamma)} [\z]_\e(\v) =0, \ \
 \forall \v \in V(\Gamma) \right\}.$$

 Note that if $\z \in X_K(\Gamma)$, then $[\z]_\e(\v) =0$ for all $\v \in \partial V(\Gamma)$. Therefore,
for $u \in BV(\Gamma)$ and $\z \in X_K(\Gamma)$, we have the following {\it Green's formula}
\begin{equation}\label{Kintbpart}
\int_{\Gamma} \z Du + \int_\Gamma u\z^{\prime} = \sum_{\v \in {\rm int}(V(\Gamma))}   \sum_{\e\in
\E_\v(\Gamma)} [\z]_\e(\v) [u]_\e(\v).
 \end{equation}

 Now, for $\v \in {\rm int}(V(\Gamma))$, we have
 $$\sum_{\e \in  E_\v(\Gamma)} [\z]_\e(\v) [u]_{\hat{\e}}(\v) =0, \quad \hbox{for all} \ \hat{\e} \in
 E_\v(\Gamma).$$
 Hence
 $$\sum_{\e\in \E_\v(\Gamma)} [\z]_\e(\v) [u]_\e(\v) = \frac{1}{d_\v} \sum_{\hat{\e}\in \E_\v(\Gamma)}
 \sum_{\e \in  E_\v(\Gamma)}[\z]_\e(\v)  \left( [u]_\e(\v) - [u]_{\hat{\e}}(\v) \right).$$
 Therefore, we can rewrite the Green's formula \eqref{Kintbpart} as

 \begin{equation}\label{KintbpartGood}
\int_{\Gamma} \z Du + \int_\Gamma u\z^{\prime} = \sum_{\v \in {\rm int}(V(\Gamma))} \frac{1}{d_\v}
\sum_{\hat{\e}\in \E_\v(\Gamma)} \sum_{\e \in  E_\v(\Gamma)}[\z]_\e(\v)  \left( [u]_\e(\v) -
[u]_{\hat{\e}}(\v) \right).
 \end{equation}

\begin{remark}{\rm
Given a function $u$ in the metric graph $\Gamma$, we say that $u$ is {\it continuous at the vertex} $\v$
if
$$[u]_{\e_1}(\v)  = [u]_{\e_2}(\v) \quad \hbox{for all} \ \e_1, \e_2 \in E_\v(\Gamma).$$
We denote this common value as $u(\v)$. We denote by $C({\rm int}(V(\Gamma)))$ the set of all functions in
$\Gamma$ continuous at the vertices $\v \in {\rm int}(V(\Gamma))$

Note that if $u \in BV(\Gamma) \cap C({\rm int}(V(\Gamma)))$ and $\z \in X_K(\Gamma)$, then by
\eqref{Kintbpart}, we have
\begin{equation}\label{KintbpartC}
\int_{\Gamma} \z Du + \int_\Gamma u\z^{\prime} =0.
 \end{equation}
 $\blacksquare$
}
\end{remark}

We  can now give our concept of total variation of a function in $BV(\Gamma)$.

\begin{definition}{\rm For $u \in BV(\Gamma)$, we define its {\it total variation} as
\begin{equation}\label{Form1}
TV_\Gamma(u) =   \sup \left\{ \displaystyle \left\vert \int_{\Gamma} u(x) \z^{\prime}(x) dx \right\vert \
: \ \z \in X_K(\Gamma), \ \Vert \z \Vert_{L^\infty(\Gamma)} \leq 1 \right\}.
\end{equation}

We say that a measurable set $E \subset \Gamma$ is a {\it set of  finite perimeter} if $\1_E \in
BV(\Gamma)$, and we define its {\it $\Gamma$-perimeter} as
$${\rm Per}_\Gamma (E):= TV_\Gamma (\1_E),$$
that is
\begin{equation}\label{forper}{\rm Per}_\Gamma (E) = \sup \left\{ \displaystyle \left\vert\int_{E}
\z^{\prime}(x) dx  \right\vert \ : \ \z \in X_K(\Gamma), \ \Vert \z \Vert_{L^\infty(\Gamma)} \leq 1
\right\}.
\end{equation}
}
\end{definition}

As  a consequence of the above definition, we have the following result.

\begin{proposition}\label{lsc1} $TV_\Gamma$ is lower semi-continuous   with  respect to the  convergence
in $L^2(\Gamma)$.
\end{proposition}

As in the local case, we have the following coarea formula relating the total variation of a function with
the perimeter of its superlevel sets.

\begin{theorem}[\bf Coarea formula]
 For any $u \in L^1(\Gamma)$, let $E_t(u):= \{ x \in \Gamma \ : \ u(x) > t \}$. Then,
\begin{equation}\label{coaerea}
TV_\Gamma(u) = \int_{-\infty}^{+\infty} {\rm Per}_\Gamma(E_t(u))\, dt.
\end{equation}
\end{theorem}

\begin{proof} We have
\begin{equation}\label{cooare1}
u(x) = \int_0^{+\infty} \1_{E_t(u)}(x) \, dt - \int_{-\infty}^0 (1 - \1_{E_t(u)}(x)) \, dt.
\end{equation}
Given $\z \in X_K(\Gamma)$ with $\Vert \z \Vert_{L^\infty(\Gamma)} \leq 1$,  since by Green's formula
\eqref{Kintbpart}
$$\int_\Gamma \z' =0,$$
and having in mind \eqref{forper}, we get
$$\int_{\Gamma} u(x) \z^{\prime}(x) dx = \int_{\Gamma}  \left(\int_{-\infty}^{+\infty} \1_{E_t(u)}(x) \,
dt\right) \z^{\prime}(x) dx $$ $$= \int_{-\infty}^{+\infty} \int_{\Gamma} \1_{E_t(u)}(x)  \z^{\prime}(x)
dx dt \leq \int_{-\infty}^{+\infty} {\rm Per}_\Gamma(E_t(u))\, dt.$$
Therefore, by \eqref{Form1}, we obtain that
$$TV_\Gamma(u) \leq \int_{-\infty}^{+\infty} {\rm Per}_\Gamma(E_t(u))\, dt.$$

To prove the other inequality, we can assume that $TV_\Gamma(u) < \infty$ and, consequently, $u \in
BV(\Gamma)$. Then, we can find a sequence $u_n \in  C^{\infty}(\Gamma)$, such that $u_n \to u$ in
$L^1(\Gamma)$ and $$\int_\Gamma \vert u'_n (x) \vert dx \to \vert Du \vert (\Gamma).$$
Now, taking a subsequence if necessary, we also have that $\1_{E_t(u_n)} \to \1_{E_t(u)}$ in $L^1(\Gamma)$
for almost all $t \in \R$. Then, by the lower semi-continuity of  ${\rm Per}_\Gamma$ and using the coarea
formula for Lipschitz functions, we have
     $$\int_{-\infty}^{+\infty} {\rm Per}_\Gamma(E_t(u))\, dt \leq \int_{-\infty}^{+\infty} \liminf_{n \to
     \infty} {\rm Per}_\Gamma(E_t(u_n))\, dt $$ $$\leq \liminf_{n \to \infty} \int_{-\infty}^{+\infty}
     {\rm Per}_\Gamma(E_t(u_n))\, dt = \lim_{n \to \infty} \int_\Gamma \vert u'_n (x) \vert dx =  \vert Du
     \vert (\Gamma) \leq TV_\Gamma(u).$$

\end{proof}

We introduce now

$$JV_\Gamma (u):=  \sum_{\v \in {\rm int}(V(\Gamma))}  \frac{1}{d_{\v}} \sum_{\hat{\e} \in E_\v(\Gamma)}
\sum_{\e \in E_\v(\Gamma)} \vert [u]_{\e}(\v) - [u]_{\hat{\e}}(\v) \vert.$$

Note that $JV_\Gamma (u)$ measures, in a weighted way, the jumps of $u$ at the vertices.

 \begin{proposition} For $u \in BV(\Gamma)$, we have
\begin{equation}\label{NForm1}
 \vert Du \vert(\Gamma) \leq TV_\Gamma(u) \leq \vert Du \vert(\Gamma) + JV_\Gamma (u).
\end{equation}

 If $u \in BV(\Gamma) \cap C({\rm int}(V(\Gamma)))$, then
 \begin{equation}\label{Form1zero}
TV_\Gamma(u) =\vert Du \vert(\Gamma).
\end{equation}

If $\Gamma$ is linear, that is $d_\v =2$ for all  $\v \in{\rm int}(V(\Gamma))$, then
  \begin{equation}\label{Form1Igual}
TV_\Gamma(u) = \vert Du \vert(\Gamma) + JV_\Gamma (u).
\end{equation}
 \end{proposition}
\begin{proof}  The inequality $\vert Du \vert(\Gamma) \leq TV_\Gamma(u)$ is  a consequence of Proposition
\ref{TVreg}. Let $u \in BV(\Gamma)$. Given $\z \in X_K(\Gamma)$ with \ $\Vert \z \Vert_{L^\infty(\Gamma)}
\leq 1$, applying Green's formula \eqref{KintbpartGood} and \eqref{booundN}, we have
$$
 \left\vert  \int_\Gamma u\z^{\prime} \right\vert = \left\vert -\int_{\Gamma} \z Du +  \sum_{\v \in {\rm
 int}(V(\Gamma))}  \frac{1}{d_{\v}}\sum_{ \hat{\e} \in E_\v(\Gamma)} \sum_{\e \in E_\v(\Gamma)} [\z]_{\e}
 (\v)([u]_{\e}(\v) - [u]_{\hat{\e}}(\v)) \right\vert$$ $$ \leq \vert Du \vert(\Gamma) +  \sum_{\v \in {\rm
 int}(V(\Gamma))}  \frac{1}{d_{\v}} \sum_{\hat{\e} \in E_\v(\Gamma)} \sum_{\e \in E_\v(\Gamma)} \vert
 [u]_{\e}(\v) - [u]_{\hat{\e}}(\v) \vert = \vert Du \vert(\Gamma) + JV_\Gamma (u).$$
Therefore,
$$ TV_\Gamma(u) = \sup \left\{\left\vert  \int_{\Gamma} u(x) \z^{\prime}(x) d(x) \right\vert \ : \ \z \in
X_K(\Gamma), \ \Vert \z \Vert_{L^\infty(\Gamma)} \leq 1 \right\} \leq \vert Du \vert(\Gamma) + JV_\Gamma
(u).   $$

Suppose now that  $u \in BV(\Gamma) \cap C({\rm int}(V(\Gamma)))$. Since $ JV_\Gamma (u) =0$, by
\eqref{NForm1}, we have
$$TV_\Gamma(u) \leq \vert Du \vert(\Gamma)$$
On the other hand,
\begin{equation}\label{TVFormula1}\vert Du \vert(\Gamma)= \sum_{\e\in \E(\Gamma)} \vert D [u]_{\e}
\vert(0,\ell_{\e}) = \sum_{\e\in \E(\Gamma)} \sup \left\{\int_0^{\ell_{\e}} [u]_\e \varphi^{\prime}_\e  \
: \  \varphi_\e \in C^\infty_c(]0,\ell_{\e}[ ), \ \Vert \varphi_\e \Vert_\infty \leq 1 \right\}.
\end{equation}
Then, since $\mathcal{D}(\Gamma) \subset X_K(\Gamma)$, we have $\vert Du \vert(\Gamma) \leq  TV_\Gamma(u)$
and \eqref{Form1zero} holds.

Finally, let us see that \eqref{Form1Igual} holds. By \eqref{TVFormula1}, for any $n \in \N$, we have that
there exists $\varphi^n_\e \in C^\infty_c((0,\ell_{\e}) )$, $\Vert \varphi^n_\e \Vert_\infty \leq 1$
\begin{equation}\label{int1}
\vert Du \vert(\Gamma) \leq \int_0^{\ell_{\e}} [u]_\e (\varphi^n_\e)^{\prime} - \frac1n.
\end{equation}

Let ${\rm supp}(\varphi^n_\e) = [a^n_\e, b^n_\e ]$, $0 <a^n_\e< b^n_\e < \ell_\e$. Now, given $\v \in {\rm
int}(V(\Gamma))$ and $\e \in E_\v (\Gamma)$, suppose that $\v = \vf_\e$ and $\vi_\e \not\in {\rm
int}(V(\Gamma))$. Then, we make the following definition: for $n \in \N$ such that $\ell_\e - \frac1n >
b^n_\e $,  $$\phi^n_\e(x):= \left\{ \begin{array}{ll}0, \quad &\hbox{if} \ 0 \leq x \leq\ell_\e - \frac1n
\\[10pt] -nx + n \ell_\e -1, \quad &\hbox{if} \  \ell_\e - \frac1n < x < \ell_\e. \end{array} \right.$$
Suppose now that $\v = \vi_\e $ and $\vf_\e \not\in {\rm int}(V(\Gamma))$.  In this case, we define, for
$n \in \N$ such that $\frac1n < a^n_\e$,
$$\phi^n_\e(x):= \left\{ \begin{array}{ll} -nx +1, \quad &\hbox{if} \ 0 \leq x \leq\ \frac1n \\[10pt] 0,
\quad &\hbox{if} \  \frac1n <x < \ell_\e. \end{array} \right.$$
Finally, suppose that $\v = \vf_\e $ and $\vi_\e \in {\rm int}(V(\Gamma))$. Then, we define, for $n \in
\N$, such that  $\frac1n < a^n_\e$ and  $\ell_\e - \frac1n > b^n_{\hat{\e}} $,
$$\phi^n_\e(x):= \left\{ \begin{array}{lll}-nx +1, \quad &\hbox{if} \ 0 \leq x \leq\ \frac1n  \\[10pt] 0,
\quad &\hbox{if} \ \frac1n < x < \ell_\e - \frac1n \\[10pt]  -nx + n \ell_\e -1, \quad &\hbox{if} \
\ell_\e - \frac1n < x < \ell_\e. \end{array} \right.$$

 Then, since $d_\v =2$ for all  $\v \in{\rm int}(V(\Gamma))$, if we define $\z^n$ such that $[\z^n]_\e:=
 \varphi^n_\e \pm \phi^n_\e$, taking the  sign of $\phi^n_\e$ depending on the orientation of $\e$, we
 have $\z^n \in X_K(\Gamma)$, and
$$\int_{\Gamma} u(x) (\z^n)^{\prime}(x) d(x) = \sum_{\e \in E(\Gamma)} \int_0^{\ell_\e} [u]_\e
[\z^n]_\e^{\prime} dx = \sum_{\e \in E(\Gamma)} \int_0^{\ell_\e} [u]_\e (\varphi^n_\e)^{\prime} dx
\pm\sum_{\e \in E(\Gamma)} \int_0^{\ell_\e} [u]_\e (\phi^n_\e)^{\prime} dx.$$

See the next Example for the definition of $\phi^n_\e$ in a particular case.

 Hence, we get
$$ \sup \left\{ \int_{\Gamma} u(x) \z^{\prime}(x) d(x)  \ : \ \z \in X_K(\Gamma), \ \Vert \z
\Vert_{L^\infty(\Gamma)} \leq 1 \right\} $$ $$\geq  \left\{ \int_{\Gamma} u(x) (\z^n)^{\prime}(x) d(x)  \
: \ n \in \N \right\} $$ $$=  \sum_{\e\in \E(\Gamma)} \int_0^{\ell_{\e}} [u]_\e )\varphi^n)^{\prime}_\e
\pm \sum_{\e \in E(\Gamma)} \int_0^{\ell_\e} [u]_\e (\phi^n_\e)^{\prime} dx  $$ $$\geq \vert Du
\vert(\Gamma) + \frac1n  \pm \sum_{\e \in E(\Gamma)} \int_0^{\ell_\e} [u]_\e (\phi^n_\e)^{\prime} dx.$$
Now,
$$\int_0^{\ell_\e} [u]_\e (\phi^n_\e)^{\prime} dx = \left\{ \begin{array}{ll} \pm n
\displaystyle\int_0^{\frac1n}[u]_\e dx \\[10pt] \pm n \displaystyle\int_{\ell_\e - \frac1n}^{\ell_\e}
[u]_\e dx. \end{array} \right. $$
Hence,
$$\lim_{n \to \infty} \int_0^{\ell_\e} [u]_\e (\phi^n_\e)^{\prime} dx = \left\{ \begin{array}{ll} \pm
[u]_\e(\vf_\e) \\[10pt] \pm [u]_\e(\vi_\e). \end{array} \right. $$
Therefore,
$$\lim_{n \to \infty} \pm \sum_{\e \in E(\Gamma)} \int_0^{\ell_\e} [u]_\e (\phi^n_\e)^{\prime} dx =
\sum_{\v \in {\rm int}(V(\Gamma))} \sum_{\e, \hat{\e} \in E_\v(\Gamma)} \vert [u]_{\e}(\v) -
[u]_{\hat{\e}}(\v) \vert.$$
Consequently, taking limit as $n \to \infty$, we obtain that
$$ \sup \left\{ \int_{\Gamma} u(x) \z^{\prime}(x) d(x)  \ : \ \z \in X_K(\Gamma), \ \Vert \z
\Vert_{L^\infty(\Gamma)} \leq 1 \right\} \geq \vert Du \vert(\Gamma) + JV_\Gamma(u) = TV_\Gamma (u).$$

\end{proof}

\begin{corollary} For $u \in BV(\Gamma)$, we have
\begin{equation}\label{conssttLinear}TV_\Gamma (u) =0 \ \iff \ u \ \hbox{is constant}.
\end{equation}
Then
\begin{equation}\label{conssttLinear1}{\rm Per}_\Gamma(E) =0 \ \iff \ E = \Gamma.
\end{equation}
\end{corollary}
\begin{proof} Obviously, if $u$ is constant, then $TV_\Gamma(u) =0$. Suppose that $TV_\Gamma(u) =0$. By
\eqref{NForm1}, we have $\vert Du \vert(\Gamma) =0.$ Then, $[u]_\e =a_\e$ is constant for all $\e \in
E(\Gamma)$. Suppose that $u$ is not constant, then there exist $\e_1,\e_2 \in E(\Gamma)$, with $a_{\e_1}
\not= a_{\e_2}$. We have
$$ TV_\Gamma(u) = \sup \left\{\left\vert  \int_{\Gamma} u(x) \z^{\prime}(x) d(x) \right\vert \ : \ \z \in
X_K(\Gamma), \ \Vert \z \Vert_{L^\infty(\Gamma)} \leq 1 \right\} $$ $$= \sup \left\{\left\vert \sum_{\e
\in E(\Gamma)} a_\e  ([\z_\e](\vf_\e)+ [\z_\e](\vi_\e)) \right\vert \ : \ \z \in X_K(\Gamma), \ \Vert \z
\Vert_{L^\infty(\Gamma)} \leq 1 \right\}.$$
We can assume that $\v = \vf_{\e_1} = \vi_{\e_2} \in {\rm int}(V(\Gamma))$. Then if we take $\z \in
W^{1,2}(\Gamma)$ such that
$$\left\{ \begin{array}{ll} [\z_{\e_1}] (\v) =1, [\z_{\e_2}] (\v) =-1,\quad \hbox{and} \ [\z_{\e}] (\v)
=0, \ \hbox{for} \ \e \not = \e_i, \ i=1,2, \\[10pt] [\z]_\e(\w) = 0, \hbox{for} \ \w \not =\v \ \hbox{and
all} \ \e \in E(\Gamma), \end{array} \right.$$
we have that $\z\in X_K(\Gamma)$ and $\Vert \z \Vert_{L^\infty(\Gamma)} \leq 1$. Therefore
$$TV_\Gamma(u) \geq \left\vert \sum_{\e \in E(\Gamma)} a_\e  ([\z_\e](\vf_\e)+ [\z_\e](\vi_\e))
\right\vert = \vert a_{\e_1} - a_{\e_2}\vert >0,$$
which is a contradiction and consequently $u$ is constant.
 \end{proof}

\begin{example}{\rm Consider the linear metric graph $\Gamma$ with four  vertices and three  edges,
$V(\Gamma) = \{\v_1, \v_2, \v_3, \v_4 \}$ and $E(\Gamma) = \{ \e_1:=[\v_1, \v_2], \e_2:=[\v_3, \v_2],
\e_3:=[\v_3, \v_4]   \}$.
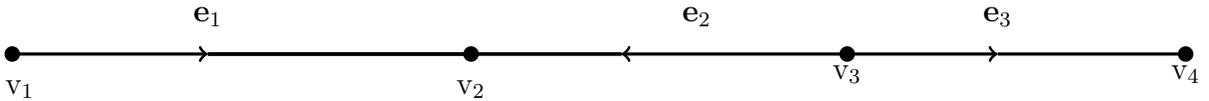
\begin{center}
\begin{tikzpicture}
  \tikzstyle{gordito2} = [line width=3]
\node (v1) at (-5,-1.2) {};
\node (v5) at (1,-1.2) {};
\node[below] at (v1) {$\v_1$};
\node[below] at (v5) {$\v_2$};

\draw[->,line width=1.2]  (-5.1,-1)--(-2.5,-1);
\draw[line width=1.2]  (-2.5,-1) --(1,-1);
\draw[line width=1.2]  (3,-1)--(-2.5,-1);
\draw[->,line width=1.2]  (6,-1)--(3,-1);
\draw[line width=1.2]  (3,-1)--(-2.5,-1);
\draw[->,line width=1.2]  (6,-1)--(8,-1);
\draw[line width=1.2]  (8,-1) --(10.5,-1);

\node at (-2.5,-0.5) {$\e_1$};
\node at (0,-1) {};
\node at (0,-1) {};
\node at (0,-1) {};
\node at (0,-1) {};
\node at (0,-1) {};
\node at (0,-1) {};
\node at (1,-1) {};
\node at (1,-1) {};

\node at (6,-1) {};
\node at (6,-1) {};
\node at (10.5,-1) {};
\node at (-4,1) {};
\node at (4,-0.5) {$\e_2$};
\node at (-4,1) {};
\node[below] at (6,-1) {$\v_3$};
\node[below] at (10.5,-1) {$\v_4$};
\node at (1,2) {};
\node at (8,-0.5) {$\e_3$};
\draw[gordito2]  (-5.1,-1) node (v2) {} circle (0.05);
\draw[gordito2]  (1,-1) node (v4) {} circle (0.05);
\draw[gordito2]  (6,-1) node (v2) {} circle (0.05);
\draw[gordito2]  (10.5,-1) node (v2) {} circle (0.05);
\end{tikzpicture}
\end{center}

Let $u \in BV(\Gamma)$ and suppose that
$$[u]_{\e_2}(\v_2) \geq   [u]_{\e_1}(\v_2) \quad \hbox{and} \quad [u]_{\e_3}(\v_3) \geq
[u]_{\e_2}(\v_3).$$  For $n \in \N$ large enough, we define
 $$\phi^n_{\e_1}(x):= \left\{ \begin{array}{ll}0, \quad &\hbox{if} \ 0 \leq x \leq\ell_{\e_1} - \frac1n
 \\[10pt] -nx + n \ell_{\e_1} -1, \quad &\hbox{if} \  \ell_{\e_1} - \frac1n < x < \ell_{\e_1}, \end{array}
 \right.$$
 $$\phi^n_{\e_2}(x):= \left\{ \begin{array}{lll}-nx +1, \quad &\hbox{if} \ 0 \leq x \leq\ \frac1n
 \\[10pt] 0, \quad &\hbox{if} \ \frac1n < x < \ell_{\e_2} - \frac1n \\[10pt]  nx - n \ell_{\e_2} +1, \quad
 &\hbox{if} \  \ell_{\e_2} - \frac1n < x < \ell_{\e_2}, \end{array} \right.$$
and
$$\phi^n_{\e_3}(x):= \left\{ \begin{array}{ll} nx -1, \quad &\hbox{if} \ 0 \leq x \leq\ \frac1n \\[10pt]
0, \quad &\hbox{if} \  \frac1n <x < \ell_{\e_3}. \end{array} \right.$$
Then, we have
$$[\z^n]_{\e_1}(\v_2)=  \phi^n_{\e_1}(\ell_{\e_1}) = -1, \quad [\z^n]_{\e_2}(\v_2)=
\phi^n_{\e_2}(\ell_{\e_2}) =1 \ \Rightarrow \ [\z^n]_{\e_1}(\v_2) + [\z^n]_{\e_2}(\v_2) =0, $$
and
$$[\z^n]_{\e_2}(\v_3)=  \phi^n_{\e_2}(0) = 1, \quad [\z^n]_{\e_3}(\v_3)=  \phi^n_{\e_3}(0) =-1 \
\Rightarrow \ [\z^n]_{\e_2}(\v_2) + [\z^n]_{\e_2}(\v_3) =0. $$
Thus, $\z^n \in X_K(\Gamma)$. Moreover,
$$\lim_{n \to \infty}\int_0^{\ell_{\e_1}} [u]_{\e_1} (\phi^n_{\e_1})^{\prime} dx = \lim_{n \to \infty}
\int_{\ell_{\e_1} -\frac1n}^{\ell_{\e_1}} (-n)[u]_{\e_1} = -  [u]_{\e_1}(\v_2),$$
$$\lim_{n \to \infty}\int_0^{\ell_{\e_2}} [u]_{\e_2} (\phi^n_{\e_2})^{\prime} dx = \lim_{n \to \infty}
\int_0^{\frac1n} (-n) [u]_{\e_2} + \lim_{n \to \infty}\int_{\ell_{\e_2} -\frac1n}^{\ell_{\e_2}}
n[u]_{\e_2}  =  - [u]_{\e_2}(\v_3) + [u]_{\e_2}(\v_2) ,$$
and
$$\lim_{n \to \infty}\int_0^{\ell_{\e_3}} [u]_{\e_3} (\phi^n_{\e_3})^{\prime} dx = \lim_{n \to \infty}
\int_0^{\frac1n} n [u]_{\e_3}  = [u]_{\e_3}(\v_3). $$
Therefore,
$$\lim_{n \to \infty} \sum_{\e \in E(\Gamma)} \int_0^{\ell_\e} [u]_\e (\phi^n_\e)^{\prime} dx = -
[u]_{\e_1}(\v_2)  - [u]_{\e_2}(\v_3) + [u]_{\e_2}(\v_2) + [u]_{\e_3}(\v_3)$$ $$ = \left([u]_{\e_2}(\v_2)-
[u]_{\e_1}(\v_2) \right)  + \left([u]_{\e_3}(\v_3)  - [u]_{\e_2}(\v_3)\right) = JV_\Gamma (u).$$
}
\end{example}

In the next example we will see that the equality \eqref{Form1Igual} does not hold if $ u \not\in C({\rm
int}(V(\Gamma)))$ or there exists $\v \in{\rm int}(V(\Gamma))$ with $d_\v \geq 3$.

\begin{example}\label{example2}{\rm
Consider the metric graph $\Gamma$ with four vertices and three edges, $V(\Gamma) = \{\v_1, \v_2, \v_3,
\v_4 \}$ and $E(\Gamma) = \{ \e_1:=[\v_1, \v_2], \e_2:=[\v_2, \v_3], \e_3:=[\v_2, \v_4] \}$.

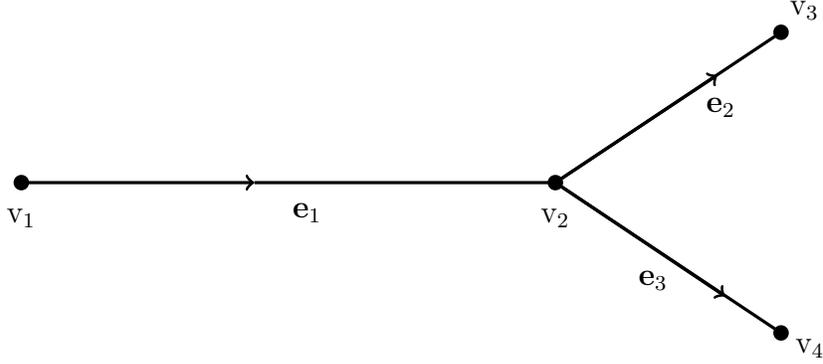
\begin{center}
\begin{tikzpicture}
 \tikzstyle{gordito2} = [line width=3]

\node (v1) at (-5.1,-1.2) {};
\node (v5) at (2,-1.2) {};

\node[below] at (v1) {$\v_1$};
\node[above] at (5.3093,1.0278) {$\v_3$};
\node[below] at (v5) {$\v_2$};
\node[above] at (5.3843,-3.4747) {$\v_4$};

\draw[gordito2]  (-5.1,-1) node (v2) {} circle (0.05);
\draw[gordito2]  (5,1) node (v3) {} circle (0.05);
\draw[gordito2]  (2,-1) node (v4) {} circle (0.05);
\draw[gordito2]  (5,-3) node (v4) {} circle (0.05);

\draw[->,line width=1.2]  (-5.1,-1)--(-2,-1);
\draw[line width=1.2]  (-2,-1) --(2,-1);
\draw[line width=1.2] (2,-1)--(5,-3);
\draw[->,line width=1.2] (2,-1)--(4.253,-2.5103);

\draw[line width=1.2](2,-1)--(5,1);

\draw[->,line width=1.2](2,-1)--(4.1524,0.4299);

\node at (-1.3,-1.4) {$\e_1$};
\node at (4.2,0) {$\e_2$};
\node at (3.3,-2.3) {$\e_3$};

\end{tikzpicture}
\end{center}

Let $ u:= \1_{\e_1} -  \1_{\e_2}$. Then,

$$JV_\Gamma (u):=  \sum_{\v \in {\rm int}(V(\Gamma))} \frac{1}{d_{\v}} \sum_{\e, \hat{\e} \in
E_\v(\Gamma)} \vert [u]_{\e}(\v) - [u]_{\hat{\e}}(\v) \vert$$ $$= \frac23 \left(\vert [u]_{\e_1}(\v_2) -
[u]_{\e_2}(\v_2) \vert + \vert [u]_{\e_1}(\v_2) - [u]_{\e_3}(\v_2) \vert + \vert [u]_{\e_2}(\v_2) -
[u]_{\e_3}(\v_2) \vert\right) = \frac83.$$

By Green's formula \eqref{KintbpartGood}, we have
$$TV_\Gamma (u) = \sup \left\{ \left\vert \int_{\Gamma} u(x) \z^{\prime}(x) d(x)  \right\vert  \ : \ \z
\in X_K(\Gamma), \ \Vert \z \Vert_{L^\infty(\Gamma)} \leq 1 \right\}$$
$$
= \sup \left\{  \left\vert \sum_{\v \in {\rm int}(V(\Gamma))}  \left(\frac{1}{d_{\v}} \right) \sum_{\e,
\hat{\e} \in E_\v(\Gamma)} [\z]_{\e} (\v)([u]_{\e}(\v) - [u]_{\hat{\e}}(\v)) \right\vert   \ : \ \z \in
X_K(\Gamma), \ \Vert \z \Vert_{L^\infty(\Gamma)} \leq 1 \right\}$$  $$
= \sup \left\{  \left\vert  \left(\frac{1}{3} \right) \sum_{\e, \hat{\e} \in E_{\v_2}(\Gamma)} [\z]_{\e}
(\v)([u]_{\e}(\v_2) - [u]_{\hat{\e}}(\v_2)) \right\vert   \ : \ \z \in X_K(\Gamma), \ \Vert \z
\Vert_{L^\infty(\Gamma)} \leq 1 \right\}.$$

Now, given $\z \in X(\Gamma)$ with  $\Vert \z \Vert_{L^\infty(\Gamma)} \leq 1$, we have $[\z]_{\e_1}
(\v_2) = [\z]_{\e_2} (\v_2) + [\z]_{\e_3} (\v_2).$ Hence,
$$\left\vert  \left(\frac{1}{3} \right) \sum_{\e, \hat{\e} \in E_{\v_2}(\Gamma)} [\z]_{\e}
(\v)([u]_{\e}(\v_2) - [u]_{\hat{\e}}(\v_2)) \right\vert $$  $$= \frac{1}{3} \Big\vert   [\z]_{\e_1}
(\v)([u]_{\e_1}(\v_2) - [u]_{\e_2}(\v_2))  +  [\z]_{\e_1} (\v)([u]_{\e_1}(\v_2) - [u]_{\e_3}(\v_2)) +
[\z]_{\e_2} (\v)([u]_{\e_2}(\v_2) - [u]_{\e_1}(\v_2)) $$ $$+ [\z]_{\e_2} (\v)([u]_{\e_2}(\v_2) -
[u]_{\e_3}(\v_2)) +  [\z]_{\e_3} (\v)([u]_{\e_3}(\v_2) - [u]_{\e_1}(\v_2)) +  [\z]_{\e_3}
(\v)([u]_{\e_3}(\v_2) - [u]_{\e_2}(\v_2))\Big\vert$$ $$= \frac{1}{3} \Big\vert 3  [\z]_{\e_1} -
[\z]_{\e_2} \Big\vert = \frac{1}{3} \Big\vert 2  [\z]_{\e_2} +3 [\z]_{\e_3} \Big\vert \leq \frac{5}{3}.$$
Therefore,
$$\frac23 \leq TV_\Gamma (u) \leq \frac{5}{3}  < \frac{8}{3} = JV_\Gamma(u).$$

}
\end{example}

\subsection{The total variation flow in metric graphs.}
In order to study the total variation flow in the metric graph $\Gamma$ we consider the energy functional
$\mathcal{F}_\Gamma : L^2(\Gamma) \rightarrow [0, + \infty]$ defined by
$$\mathcal{F}_\Gamma(u):= \left\{ \begin{array}{ll} \displaystyle
TV_\Gamma(u)
 \quad &\hbox{if} \ u\in  BV(\Gamma), \\[10pt] + \infty \quad &\hbox{if } u\in L^2(\Gamma)\setminus
 BV(\Gamma), \end{array} \right.$$
which is convex and lower semi-continuous.
 Following the method used in \cite{ACMBook} we will characterize the subdifferential of the functional
 $\mathcal{F}_\Gamma$.

Given a functional $\Phi : L^2(\Gamma) \rightarrow [0, \infty]$, we define
$\widetilde {\Phi}: L^2(\Gamma) \rightarrow [0, \infty]$ as
\begin{equation}\label{rrt}\widetilde {\Phi}(v):= \sup \left\{ \frac{\displaystyle \int_{\Gamma} v(x) w(x)
d(x)}{\Phi(w)} \ : \ w \in L^2(\Gamma) \right\}\end{equation}
with the convention that $\frac{0}{0} = \frac{0}{\infty} = 0$.  Obviously, if $\Phi_1 \leq \Phi_2$, then
$\widetilde {\Phi}_2 \leq \widetilde {\Phi}_1$.

    \begin{theorem}\label{chasubd} Let $u \in  BV(\Gamma)$ and $v \in L^2(\Gamma)$. The following
    assertions are equivalent:

\noindent (i) $v \in \partial \mathcal{F}_\Gamma (u)$;

\noindent
(ii) there   exists  $\z  \in X_K(\Gamma)$, $\Vert \z \Vert_{L^\infty(\Gamma)} \leq 1$ such that
\begin{equation}\label{eqq2}
   v = -\z^{\prime}, \quad  \hbox{that is,} \quad  [v]_\e = -[\z]_\e^{\prime} \ \ \hbox{in} \
   \mathcal{D}^{\prime}(]0, \ell_\e[) \  \forall \e \in E(\Gamma)
\end{equation}
and
\begin{equation}\label{eqq1}
\int_{\Gamma} u(x) v(x) dx = \mathcal{F}_\Gamma (u);
\end{equation}

\noindent
(iii)  there exists $\z  \in X_K(\Gamma)$, $\Vert \z \Vert_{L^\infty(\Gamma)} \leq 1$  such that
\eqref{eqq2} holds and
\begin{equation}\label{eqq3}
\mathcal{F}_\Gamma (u) = \int_{\Gamma}   \z Du -  \sum_{\v \in {\rm int}(V(\Gamma))} \frac{1}{d_\v}
\sum_{\hat{\e}\in \E_\v(\Gamma)} \sum_{\e \in  E_\v(\Gamma)}[\z]_\e(\v)  \left( [u]_\e(\v) -
[u]_{\hat{\e}}(\v) \right).
\end{equation}

Moreover, $D(\partial \mathcal{F}_\Gamma)$ is dense in $L^2(\Gamma)$.
\end{theorem}

\begin{proof}
Since $\mathcal{F}_\Gamma$ is convex, lower semi-continuous and positive homogeneous of degree $1$, by
\cite[Theorem 1.8]{ACMBook}, we have
\begin{equation}\label{saii}\partial \mathcal{F}_\Gamma (u) = \left\{ v \in L^2(\Gamma)  \ : \
\widetilde{\mathcal{F}_\Gamma}(v) \leq 1, \  \int_{\Gamma} u(x) v(x) dx = \mathcal{F}_\Gamma
(u)\right\}.\end{equation}

We define, for $v \in   L^2(\Gamma)$,
\begin{equation}\label{Form3}\Psi(v):= \left\{ \begin{array}{ll} \inf \left\{ \Vert \z
\Vert_{L^\infty(\Gamma)} \ : \ \z  \in X_K(\Gamma), \ v = -  \z^{\prime} \right\} \\[10pt] + \infty \ \
\hbox{if does not exists} \ \ \z  \in X_K(\Gamma), s.t.,  \ v = -  \z^{\prime}. \end{array} \right.
\end{equation}
Observe that $\Psi$ is convex,  lower semi-continuous  and positive homogeneous of degree $1$.
Moreover,  if $\Psi(v) < \infty$, the infimum in \eqref{Form3} is attained, i.e., there exists some $\z
\in X_K(\Gamma)$ such that  $v = - \z^{\prime}$ and $\Psi(v) = \Vert \z \Vert_{L^\infty(\Gamma)}$. In
fact, let $\z_n  \in X_K(\Gamma)$ with $v = -  \z_n^{\prime}$ for all $n \in \N$, such that $\Psi(v) =
\lim_{n \to \infty} \Vert \z_n \Vert_\infty$. We can assume that
$$\lim_{n \to \infty}\z_n = \z \quad \hbox{weakly$^*$ in} \ L^\infty (\Gamma), \quad \hbox{and} \quad
\z^{\prime} =v.$$
We must show that $\z$ satisfies the Kirchhoff condition. Now,
by Proposition \ref{megusta1}, we have that
\begin{equation}\label{OOle}\lim_{n \to \infty} \int_\Gamma \z_n Du = \int_\Gamma \z Du, \quad \forall \,
u \in BV(\Gamma).
\end{equation}
Fix $\v \in V(\Gamma)$. Applying Green's formula \eqref{Kintbpart} to $\z_n$ and $u \in BV(\Gamma)$ , we
get
$$\int \z_n Du + \int_\Gamma u \z_n^{\prime} = \sum_{\v \in V(\Gamma)} \sum_{\e\in \E_\v(\Gamma)}
[\z]_\e(\v) [u]_\e(\v).$$
Hence, taking $u$ such that $[u]_e(\v) =1$ for all $\e\in \E_\v(\Gamma)$ and $[u]_{\hat{\e}} =0$ if $\v
\not\in \E_\v(\Gamma)$, we have
$$\int \z_n Du + \int_\Gamma u \z_n^{\prime} =  \sum_{\e\in \E_\v(\Gamma)} [\z_n]_\e(\v) [u]_\e(\v) =0.$$
Then, taking the limit as $n \to \infty$ and having in mind  \eqref{intbpart}, we obtain
$$0 = \int \z Du + \int_\Gamma u \z^{\prime} =  \sum_{\e\in \E_\v(\Gamma)} [\z]_\e(\v) [u]_\e(\v).$$
Therefore, $\z \in X_K(\Gamma)$ and $\Psi(v) = \Vert \z \Vert_{L^\infty(\Gamma)}$.

Let us see that $$\Psi =  \widetilde{\mathcal{F}_\Gamma}.$$
 We begin by proving that $\widetilde{\mathcal{F}_\Gamma}(v) \leq   \Psi(v)$.
If $\Psi (v) = +\infty$ then this assertion is trivial. Therefore, suppose that $\Psi (v) < +\infty$.
Given $\epsilon >0$, there exists $\z \in X_K(\Gamma)$ such that $v = - \z^{\prime}$ and $\Vert \z
\Vert_{L^\infty(\Gamma)} \leq \Psi(v) + \epsilon$. Then, for $w\in BV(\Gamma)$, applying Green's formula
\eqref{KintbpartGood} and \eqref{booundN}, we have
$$\int_{\Gamma} w(x) v(x) dx = - \int_{\Gamma} w(x) \z^{\prime}(x) dx  = \int_{\Gamma} \z Dw $$ $$
-\sum_{\v \in {\rm int}(V(\Gamma))} \frac{1}{d_\v} \sum_{\hat{\e}\in \E_\v(\Gamma)} \sum_{\e \in
E_\v(\Gamma)}[\z]_\e(\v)  \left( [u]_\e(\v) - [u]_{\hat{\e}}(\v) \right)\leq \Vert \z
\Vert_{L^\infty(\Gamma)} TV_\Gamma (w).$$
Taking the supremum over $w$ we obtain that $\widetilde{\mathcal{F}_\Gamma}(v) \leq  \Vert \z
\Vert_{L^\infty(\Gamma)} \leq \Psi(v) + \epsilon$, and since this is true for all $\epsilon >0$,  we get
$\widetilde{\mathcal{F}_\Gamma}(v) \leq   \Psi(v)$.

To prove the opposite inequality  let us denote
$$D:= \{  \z^{\prime}  \ : \ \z \in X_K(\Gamma) \}.$$
Then, by \eqref{Form1}, we have that, for $v\in L^2(\Gamma)$,
$$
\begin{array}{rl}
\displaystyle
 \displaystyle\widetilde{\Psi}(v)\!\!\!\! &\displaystyle
 = \sup_{w \in L^2(\Gamma)} \frac{\displaystyle\int_{\Gamma} w(x) v(x) dx}{\Psi(w)}   \geq \sup_{w \in D }
 \frac{\displaystyle\int_{\Gamma} w(x) v(x) dx}{\Psi(w)} \\ \\
& \displaystyle=  \sup_{\z \in  X_K(\Gamma)}  \frac{\displaystyle\int_{\Gamma} \z^{\prime}(x) v(x)
dx}{\Vert \z \Vert_{L^\infty(\Gamma)}}   =   \mathcal{F}_\Gamma(v).
\end{array}
$$
 Thus, $ \mathcal{F}_\Gamma\leq \widetilde{ \Psi}$, which implies, by \cite[Proposition 1.6]{ACMBook},
 that $\Psi = \widetilde{\widetilde{\Psi}} \leq \widetilde{ \mathcal{F}_\Gamma}$. Therefore, $\Psi =
 \widetilde{\mathcal{F}_\Gamma}$, and, consequently, from \eqref{saii}, we get
$$
\begin{array}{l}
\displaystyle
\partial \mathcal{F}_\Gamma (u) = \left\{ v \in L^2(\Gamma)  \ : \   \Psi(v) \leq 1, \  \int_{\Gamma} u(x)
v(x) dx = \mathcal{F}_\Gamma(u)\right\} \\[10pt]
  \displaystyle
\phantom{\partial \mathcal{F}_m (u)}
   = \left\{ v \in L^2(\Gamma)  \ : \ \exists \z  \in X_K(\Gamma), \ v = - \z^{\prime}, \ \Vert \z
   \Vert_{L^\infty(\Gamma)}  \leq 1, \  \int_{\Gamma} u(x) v(x) dx = \mathcal{F}_\Gamma(u)\right\},
\end{array}
$$
from where the equivalence between (i) and (ii) follows.

To prove the equivalence between (ii) and (iii) we only need to apply Green's formula
\eqref{KintbpartGood}.

Finally, by \cite[Proposition 2.11]{Brezis}, we have
$$ D(\partial \mathcal{F}_\Gamma) \subset  D(\mathcal{F}_\Gamma) =  BV(\Gamma) \subset
\overline{D(\mathcal{F}_\Gamma)}^{L^2(\Gamma)} \subset \overline{D(\partial
\mathcal{F}_\Gamma)}^{L^2(\Gamma)},$$
from which the density of the domain follows.

\end{proof}

We can also prove the following characterization of the subdifferential in terms of variational
inequalities.

\begin{proposition}
The following conditions are equivalent: \\
$(a)$ $(u,v) \in \partial\mathcal{F}_\Gamma$; \\
$(b)$ $u, v \in L^2(\Gamma)$, $u \in BV(\Gamma)$ and there exists  $\z  \in X_K(\Gamma)$, $\Vert \z
\Vert_{L^\infty(\Gamma)} \leq 1$ such that  $v = -\z^{\prime}$, and for every $w \in BV(\Gamma)$
\begin{equation}\label{eq:variationalinequalitytvflow}\begin{array}{ll}
\displaystyle\int_{\Gamma} v(w-u) \, dx \\[10pt]\leq \displaystyle\int_{\Gamma} \z Dw - \sum_{\v \in {\rm
int}(V(\Gamma))}  \left(\frac{1}{d_{\v}} \right) \sum_{\e, \hat{\e} \in E_\v(\Gamma)} [\z]_{\e}
(\v)([w]_{\e}(\v) - [w]_{\hat{\e}}(\v))- TV_\Gamma(u); \end{array}
\end{equation}
$(c)$ $u, v \in L^2(\Gamma)$, $u \in BV(\Gamma)$ and there exists $\z  \in X_K(\Gamma)$, $\Vert \z
\Vert_{L^\infty(\Gamma)} \leq 1$ such that  $v = -\z^{\prime}$,
 and for every $w \in BV(\Gamma)$
\begin{equation}
\begin{array}{ll}
\displaystyle\int_{\Gamma} v(w-u) \, dx \\[10pt] = \displaystyle\int_{\Gamma} \z Dw - \sum_{\v \in {\rm
int}(V(\Gamma))}  \left(\frac{1}{d_{\v}} \right) \sum_{\e, \hat{\e} \in E_\v(\Gamma)} [\z]_{\e}
(\v)([w]_{\e}(\v) - [w]_{\hat{\e}}(\v))- TV_\Gamma(u). \end{array}
\end{equation}
\end{proposition}

\begin{proof} $(a) \Rightarrow (c)$:  By Theorem \ref{chasubd}, we have that there exists $\z  \in
X_K(\Gamma)$, $\Vert \z \Vert_{L^\infty(\Gamma)} \leq 1$ such that  $v = -\z^{\prime}$ and
$$\mathcal{F}_\Gamma (u) = \int_{\Gamma}   \z Du - \sum_{\v \in {\rm int}(V(\Gamma))}
\left(\frac{1}{d_{\v}} \right) \sum_{\e, \hat{\e} \in E_\v(\Gamma)} [\z]_{\e} (\v)([u]_{\e}(\v) -
[u]_{\hat{\e}}(\v)).$$
Then, given $w \in BV(\Gamma)$, multiplying $v = -\z^{\prime}$ by $w -u$, integrating over $\Gamma$ and
using Green's formula \eqref{KintbpartGood}, we get
\begin{equation*}
\int_{\Gamma} v(w-u) \, dx = - \int_{\Gamma} (w - u) \z' \, dx  $$ $$= \int_{\Gamma} \z Dw - \sum_{\v \in
{\rm int}(V(\Gamma))}  \left(\frac{1}{d_{\v}} \right) \sum_{\e, \hat{\e} \in E_\v(\Gamma)} [\z]_{\e}
(\v)([w]_{\e}(\v) - [w]_{\hat{\e}}(\v))- TV_\Gamma(u).
\end{equation*}

Obviously, $(c)$ implies $(b)$. To finish the proof, let us see that $(b)$ implies $(a)$. If we take $w =
u$ in \eqref{eq:variationalinequalitytvflow}, we get
\begin{equation*}
TV_\Gamma(u) \leq \int_{\Gamma} \z Du -  \sum_{\v \in {\rm int}(V(\Gamma))}  \left(\frac{1}{d_{\v}}
\right) \sum_{\e, \hat{\e} \in E_\v(\Gamma)} [\z]_{\e} (\v)([u]_{\e}(\v) - [u]_{\hat{\e}}(\v)),
\end{equation*}
and, therefore, by \eqref{booundN}, we have
\begin{equation*}
TV_\Gamma(u) = \int_{\Gamma} \z Du -  \sum_{\v \in {\rm int}(V(\Gamma))}  \left(\frac{1}{d_{\v}} \right)
\sum_{\e, \hat{\e} \in E_\v(\Gamma)} [\z]_{\e} (\v)([u]_{\e}(\v) - [u]_{\hat{\e}}(\v)).
\end{equation*}

\end{proof}

\begin{proposition}\label{intpartt} For  any $v \in \partial \mathcal{F}_\Gamma (u)$   it holds that
\begin{equation}\label{intpart}
 \int_\Gamma v w dx \leq \mathcal{F}_\Gamma(w) \qquad \hbox{for all } \  w \in BV(\Gamma),
\end{equation}
and
 \begin{equation}\label{intpart2}
 \int_\Gamma v u dx = \mathcal{F}_\Gamma(u).
\end{equation}
\end{proposition}
\begin{proof} Since $v \in \partial \mathcal{F}_\Gamma(u)$, given $w \in BV(\Gamma)$, we have that
$$\int_\Gamma v w dx \leq \mathcal{F}_\Gamma(u+w) - \mathcal{F}_\Gamma(u) \leq \mathcal{F}_\Gamma(w),$$
so we get \eqref{intpart}.  On the other hand, \eqref{intpart2} is given in Theorem \ref{chasubd}.
\end{proof}

 \begin{definition}{\rm We define the {\it $1$-Laplacian operator} in the metric graph $\Gamma$ as
 $$(u, v ) \in \Delta_1^{\Gamma} \iff -v \in \partial \mathcal{F}_\Gamma(u),$$
 that is, if $u \in BV(\Gamma)$,  $v \in L^2(\Gamma)$ and there   exists  $\z  \in X_K(\Gamma)$,
 $\Vert \z \Vert_{L^\infty(\Gamma)} \leq 1$ such that
\begin{equation}\label{Neqq2}
   v = \z^{\prime}, \quad  \hbox{that is,} \quad  [v]_\e = [\z]_\e^{\prime} \quad   \ \hbox{in} \
   \mathcal{D}^{\prime}(]0, \ell_\e[) \  \forall \e \in E(\Gamma),
\end{equation}
and
\begin{equation}\label{1eqq3}
\mathcal{F}_\Gamma (u) = \int_{\Gamma}   \z Du  - \sum_{\v \in {\rm int}(V(\Gamma))}
\left(\frac{1}{d_{\v}} \right) \sum_{\e, \hat{\e} \in E_\v(\Gamma)} [\z]_{\e} (\v)([u]_{\e}(\v) -
[u]_{\hat{\e}}(\v).
\end{equation}
 }
 \end{definition}

 We have that the Cauchy problem
 \begin{equation}\label{CP1}
 \left\{ \begin{array}{ll} \displaystyle\frac{\partial u}{\partial t}(t) \in \Delta_1^{\Gamma}u(t) \quad
 &t \geq 0 \\ \\ u(0) = u_0  \quad  &u_0 \in L^2(\Gamma) \end{array} \right.
 \end{equation}
 can be  rewritten as the abstract Cauchy problem in $L^2(\Gamma)$,
 \begin{equation}\label{ACP1}
 \left\{ \begin{array}{ll}  u'(t) + \partial \mathcal{F}_\Gamma u(t)\ni 0 \quad  &t \geq 0 \\ \\ u(0) =
 u_0  \quad  &u_0 \in L^2(\Gamma). \end{array} \right.
 \end{equation}

 Since $\mathcal{F}_\Gamma$ is convex and lower semi-continuous in  $L^2(\Gamma)$ and $D(\partial
 \mathcal{F}_\Gamma)$  is  dense in  $L^2(\Gamma)$  by the Brezis-Komura theory (see \cite{Brezis}), we
 have that for any initial data $u_0 \in L^2(\Gamma)$ there exists a unique strong solution to problem
 \eqref{ACP1}. Therefore, we have the following existence and uniqueness result.

 \begin{theorem} For any initial data $u_0 \in L^2(\Gamma)$ there exists a unique solution
 $u(t)$ of the Cauchy problem \eqref{CP1}, in the sense that $u \in C(0, T; L^2(\Gamma)) \cap W^{1,1}(0,T;
 L^2(\Gamma))$  for any $T >0$, $u(t) \in BV(\Gamma) $  and there exists $\z \in L^\infty(0,T;
 L^\infty(\Gamma))$, $\z(t) \in X_K(\Gamma)$, $\Vert \z(t) \Vert_{L^\infty(\Gamma)} \leq 1$, for almost
 all $t \in (0,T)$, such that
\begin{equation}\label{TNeqq2}
   u'(t) = \z(t)^{\prime}, \quad  \hbox{that is,} \quad  [u(t)]^{\prime}_\e = [\z(t)]_\e^{\prime}  \
   \hbox{in} \ \mathcal{D}^{\prime}(]0, \ell_\e[) \  \forall \e \in E(\Gamma)
\end{equation}
and
\begin{equation}\label{Teqq1}
TV_\Gamma (u(t)) = \int_{\Gamma}   \z(t) Du(t) - \sum_{\v \in {\rm int}(V(\Gamma))}
\left(\frac{1}{d_{\v}} \right) \sum_{\e, \hat{\e} \in E_\v(\Gamma)} [\z(t)]_{\e} (\v)([u(t)]_{\e}(\v) -
[u(t)]_{\hat{\e}}(\v)).
\end{equation}

\end{theorem}

\begin{definition}{\rm
Given $u_0 \in L^2(\Gamma)$, we denote by $e^{t \Delta_1^{\Gamma}}u_0$ the unique solution of problem
\eqref{CP1}. We call the semigroup $\{e^{t\Delta_1^{\Gamma}} \}_{t \geq 0}$ in $L^2(X, \nu)$ the {\it
Total Variational Flow in the metric graph} $\Gamma$.
}
\end{definition}

The Total Variational Flow in the metric graph satisfies the mass conservation property.

\begin{proposition}\label{propert1}  :
For $u_0 \in L^2(\Gamma)$,  $$\int_\Gamma e^{t\Delta_1^{\Gamma} }u_0 dx = \int_\Gamma u_0 dx \quad
\hbox{for any} \ t \geq 0.$$
\end{proposition}

\begin{proof}  By (ii) in Theorem \ref{chasubd} and Green's formula \eqref{Kintbpart}, we have
$$- \frac{d}{dt} \int_\Gamma e^{t\Delta_1^{\Gamma}}u_0 dx = -\int_\Gamma \z(t)' dx = \int_\Gamma \z(t)D
\1_\Gamma \leq TV_\Gamma(\1_\Gamma) = 0, $$ and
$$  \frac{d}{dt} \int_\Gamma e^{t\Delta_1^{\Gamma}}u_0 dx
\leq TV_\Gamma(-\1_\Gamma) = 0. $$
Hence,
$$\frac{d}{dt} \int_\Gamma e^{t\Delta_1^{\Gamma}}u_0 dx =0,$$
and, consequently,
$$\int_\Gamma e^{t\Delta_1^{\Gamma}}u_0 dx = \int_\Gamma u_0 dx \quad \hbox{for any} \ t \geq 0.$$
 \end{proof}

\subsection{Asymptotic Behaviour}
By \eqref{conssttLinear}, we have $$\mathcal{N}(\mathcal{F}_\Gamma):= \{ u \in L^2(\Gamma) \ : \
\mathcal{F}_\Gamma(u)=0 \} = \{u \in L^2(\Gamma) \ : \ u \ \hbox{is constant} \}.$$

Since $\mathcal{F}_\Gamma$  is a proper and lower semicontinuous function in $L^2(\Gamma)$ attaining a
minimum at the constant zero function and,  moreover, $\mathcal{F}_\Gamma$ is even, by \cite[Theorem
5]{Bruck},we have that there exists $v_0 \in \mathcal{N}(\mathcal{F}_\Gamma)$ such that $$\lim_{t \to
\infty} e^{t\Delta_1^{\Gamma}}u_0 = v_0 \quad \hbox{in } L^2(\Gamma).$$
Now,   having in mind Proposition \ref{propert1}, we have

 $$v_0 = \overline{u_0} := \frac{1}{\ell(\Gamma)} \int_\Gamma u_0 dx.$$

We denote
$$T_{ex}(u_0):= \inf \{ T >0 \ : \ e^{t\Delta_1^{\Gamma} }u_0 = \overline{u_0},  \ \ \forall \, t \geq T
\}.$$
We will see that the total variational flow in the metric graph $\Gamma$ reaches the  mean
$\overline{u_0}$ of the initial data $u_0$ in finite time, that is, $T_{ex}(u_0) < \infty$. We will rely
on the results proved by Bungert and Burger in \cite{BB} (see also \cite{BBChN}) for the gradient flow of
a coercive (in the sense described below),  absolutely  $1$-homogeneous convex functional defined on
a Hilbert space.

Let us recall some notation used in \cite{BB}. Let $\mathcal{H}$ be a Hilbert space and $J: \mathcal{H}
\rightarrow (-\infty, + \infty]$ a proper, convex, lower semi-continuous functional. Then, it is well
known (see \cite{Brezis}) that the abstract Cauchy problem
\begin{equation}\label{ACP123}
\left\{ \begin{array}{ll} u'(t) + \partial J(u(t)) \ni 0, \quad t \in [0,T] \\[5pt] u(0) = u_0,
\end{array}\right.
\end{equation}
has a unique strong solution $u(t)$ for any initial datum $u_0 \in \overline{D(J)}$.

We say that $J$ is {\it coercive}, if there exists a constant $C >0$ such that
\begin{equation}\label{coerc}
\Vert u \Vert \leq C J(u), \quad \forall \, u\in \mathcal{H}_0,
\end{equation}
where
$$\mathcal{H}_0:= \{ u \in \mathcal{H} \ : \ J(u)=0 \}^{\perp} \setminus \{ 0 \}.$$
Clearly, this inequality is equivalent to positive lower bound of the {\it Rayleigh quotient} associated
with $J$, i.e.,
$$\lambda_1(J):= \inf_{u \in \mathcal{H}_0} \frac{J(u)}{\Vert u \Vert} > 0.$$

For $u_0 \in \mathcal{H}_0$, if $u(t)$ is the strong solution of \eqref{ACP123}, we define its {\it
extinction time} as
$$T_{\rm ex}(u_0):= \inf \{ T >0 \ : \ u(t) =0,  \ \ \forall \, t \geq T \}.$$

In the next result, we summarize the results obtained by Bungert and Burger in \cite{BB}.

\begin{theorem}\label{mainResult} Let $J$ be a convex, lower-semicontinuous functional on  $\mathcal{H}$
with dense domain. Assume that $J$ is   absolutely $1$-homogeneous and coercive. For  $u_0 \in
\mathcal{H}_0$, let  $u(t)$ be the strong solution of \eqref{ACP123}. Then, we have
\begin{itemize}
\item[(i)] (Finite extinction time)
$$T_{\rm ex}(u_0) \leq \frac{\Vert u_0 \Vert}{\lambda_1(J)}.$$

\item[(ii)] (General upper bounds)
$$\Vert u(t) \Vert \leq \Vert u_0 \Vert - \lambda_1(J) t,$$

\item[(iii)] (Sharper bound for the finite extinction)
$$\lambda_1(J) (T_{\rm ex}(u_0) - t) \leq \Vert u(t) \Vert \leq \Lambda (t)(T_{\rm ex}(u_0) - t),  $$
where
  $$\Lambda(t):= \frac{J(u(t))}{\Vert u(t) \Vert}.$$
\
\end{itemize}

\end{theorem}

Now we are going to apply Theorem \ref{mainResult} to study the asymptotic behaviour of the  solutions of
the Cauchy  problem \eqref{CP1}.

Obviously, the convex, lower semi-continuous functional $\mathcal{F}_\Gamma$   is absolutely
$1$-homogeneous, that is,  $\mathcal{F}_\Gamma(\lambda u) = \vert \lambda \vert
\mathcal{F}_\Gamma(u)$, for all $\lambda \in \R$ and all $u \in L^2(\Gamma)$. In this case,
$$L^2(\Gamma)_0:= \mathcal{N}(\mathcal{F}_\Gamma)^{\perp} \setminus \{ 0 \} =   \left\{ u \in
L^2(\Gamma) \ : \ \int_\Gamma u(x) dx =0 \right\}\setminus \{ 0 \}.$$
Let us see that $\mathcal{F}_\Gamma$ is coercive. In fact, if it weren't we could find a sequence $u_n \in
L^2(\Gamma)_0$ such that
$$\Vert u_n \Vert_{L^2(\Gamma)} \geq n \mathcal{F}_\Gamma(u_n), \quad \forall \, n \in\N.$$
Now, by homogeneity, we can asume that $\Vert u_n \Vert_{L^2(\Gamma)} =1$ for all $n \in \N$, so
$$TV_\Gamma(u_n) \leq \frac1n, \quad \forall \, n \in\N.$$
By Theorem \ref{embedding}, we can assume, taking a subsequence if necessary, that
$$ u_n \to u, \quad \hbox{in} \ L^2(\Gamma).$$
Then, by the lower semi-continuity of $TV_\Gamma$ (Proposition \ref{lsc1}), we have $TV_\Gamma(u)
=0$. Then, by \eqref{conssttLinear}, $u$ is constant. Now, since $u_n \in L^2(\Gamma)_0$,
$$\int_\Gamma u_n(x) dx =0, \quad \hbox{for all } \ n \in \N.$$
Therefore, $\Vert u \Vert _{L^2(\Gamma)} =0$, which is a contradiction since $\Vert u \Vert _{L^2(\Gamma)}
=1.$

If we denote
$$\lambda_{\Gamma} := \inf \left\{ \frac{TV_\Gamma (u)}{\Vert u \Vert_{ L^2(\Gamma)}} \ : \ u \in
L^2(\Gamma)_0 \right\} >0,$$ we have
\begin{equation}\label{Poincare2}
\Vert u  \Vert_{L^2(\Gamma} \leq \lambda_{\Gamma} TV_\Gamma (u) \quad \hbox{for all} \ u \in
L^2(\Gamma)_0.
\end{equation}

 Then, by Theorem \ref{mainResult}, we have the following result.

\begin{theorem} For any $u_0 \in L^2(\Gamma)$, we have
\begin{equation}\label{exttin}
T_{ex}(u_0) \leq \frac{\left\Vert u_0-\overline{u_0}\right\Vert_{L^2(\Gamma)}}{\lambda_{\Gamma}}.
\end{equation}
Moreover,
\begin{equation}
        \lambda_{\Gamma}(T_{ex}(u_0) - t) \leq  \Vert u(t)-\overline{u_0}\Vert_{L^2(\Gamma)}\le
        \Lambda(t)(T_{ex}(u_0) - t),
    \end{equation}
       where
       $$ \Lambda(t):= \frac{\mathcal{F}_{\Gamma} ( u(t))}{\Vert
       u(t)-\overline{u_0}\Vert_{L^2(\Gamma)}}.$$

\end{theorem}
  \begin{proof} It is a direct application of Theorem \ref{mainResult}, having in mind that for any
  constant function $v_0$ and any $u_0 \in L^2(\Gamma)$, we have   $\mathcal{F}_{\Gamma}(u_0 +
  \overline{u_0}) = \mathcal{F}_{\Gamma}(u_0)$ and  $\partial \mathcal{F}_{\Gamma}(u_0 + \overline{u_0}) =
  \partial \mathcal{F}_{\Gamma}(u_0)$ (see \cite[Proposition A.3]{BB}).

\end{proof}

 To obtain a  lower bound on the extinction time, we introduce the following space which,  in the
 continuous setting, was introduced in \cite{Meyer}:
 $$G_m(\Gamma):= \{ v \in L^2(\Gamma) \ : \ \exists \z \in  X_K(\Gamma), \   v = -\z' \ \hbox{a.e.
 in} \Gamma \},$$
 and consider in $G_m(\Gamma)$ the norm
$$\Vert v \Vert_{m,*} := \inf\{\Vert \z \Vert_\infty \ : \z \in  X_K(\Gamma), \ v = -\z' \
\hbox{a.e. in} \ \Gamma \}.$$

Note that, for $v \in G_m(\Gamma)$, we have that there exists $\z_v\in X(\Gamma)$, such that $v = -\z'_v$
and $\Vert v \Vert_{m,*} = \Vert \z_v \Vert_\infty$.

From the proof of Theorem \ref{chasubd}, for $f \in G_m(\Gamma)$, we have

 \begin{equation}\label{1meyer1}
 \Vert f \Vert_{m,*}:= \sup \left\{ \left\vert \int_\Gamma f(x) u(x) dx \right\vert   : u \in BV(\Gamma),
 \ TV_{\Gamma}(u) \leq 1\right\},
 \end{equation}
 and, moreover,
\begin{equation}\label{saiiN}\partial \mathcal{F}_\Gamma (u) = \left\{ v \in L^2(\Gamma)  \ : \ \Vert v
\Vert_{m,*} \leq 1, \  \int_{\Gamma} u(x) v(x) dx = TV_\Gamma (u)\right\}.\end{equation}

 The next result is consequence of \cite[Proposition 6.9]{BBChN}. We give the proof to be
self-contained

   \begin{theorem}  Given   $u_0 \in L^2(\Gamma)$, we have
 \begin{equation}
 T_{\rm ex}(u_0) \geq \Vert u_0 - \overline{u_0}\Vert_{m,*}.
 \end{equation}
 \end{theorem}
 \begin{proof} If $u(t):= e^{t\Delta^\Gamma_1}u_0$, we have
 $$u_0 - \overline{u_0}= - \int_0^{T_{ex}(u_0)}   u'(t)dt.$$
Then, by Proposition \ref{intpartt}, we get
 $$\Vert u_0 - \overline{u_0} \Vert_{m,*} = \sup \left\{\int_\Gamma w (u_0 -\overline{u_0})  dx \ : \
 TV_m(w) \leq 1  \right\} $$ $$= \sup \left\{ \int_\Gamma w \left(  \int_0^{T_{\rm ex}(u_0)} -
 u'(t)dt\right) dx \ : \ TV_m(w) \leq 1  \right\}$$ $$= \sup \left\{ \int_0^{T_{\rm ex}(u_0)} \int_\Gamma
 -w    u'(t)dt  dx \ : \ TV_m(w) \leq 1  \right\}$$  $$\leq  \sup \left\{\int_0^{T_{\rm ex}(u_0)} TV_m(w)
 dt \ : \ TV_m(w) \leq 1  \right\} = T_{ex}(u_0).$$
 \end{proof}

\subsection{Explicit Solutions} Let us now see that we can compute explicitly the evolution of
characteristic functions. First we need to do the computations for the Neumann problem for the total
variation flow in an interval $(0, L)$ of $\R$, that is, for the problem
\begin{equation}\label{NeumannR}
\left\{ \begin{array}{lll} u_t = {\rm div} \left( \frac{Du}{\vert Du \vert}\right), \quad &\hbox{in} \ \
]0,T[ \times ]0,L[, \\[10pt] \frac{Du}{\vert Du \vert} \cdot \eta =0, \quad &\hbox{in} \ \ ]0,T[ \times
\{0,L \}, \\[10pt] u(0) = u_0.  \end{array} \right.
\end{equation}

In \cite{ABCM0}, we have proved the existence and uniqueness of solutions to problem \eqref{NeumannR},
where the concept of solution is the following. For $u_0 \in L^2(]0,L[)$ we say that $u \in C(0,T;
L^2(]0,L[) \cap W^{1,1}(0,T;L^2(]0,L[)$ is a weak solution of \eqref{NeumannR} if $u(0) = u_0$, $u(t) \in
BV((0,L)) $  and there exists $\z \in L^\infty(0,T; L^\infty(]0,L[)$,$\Vert \z(t) \Vert_{L^\infty(]0,L[)}
\leq 1$, for almost all $t \in ]0,T[$, such that
\begin{equation}\label{TNeqq2N}
   u'(t) = \z(t)^{\prime},\quad  \hbox{in} \ \mathcal{D}^{\prime}(]0,L[),
\end{equation}
\begin{equation}\label{frontN}
 \z(t)(0) = \z(t)(L)=0,
 \end{equation}
and
\begin{equation}\label{Teqq1N}
\int_0^L \vert Du(t)\vert  = \int_0^L   \z(t) Du(t).
\end{equation}

\begin{lemma}\label{inR} Let $0 < a,b,c < L$ and $k >0$. Then, we have
\begin{itemize}
\item[(1)]  If $u_0= k\1_{(0,a)}$, then the solution of \eqref{NeumannR} is given by
$$u(t)= \left( k - \frac{t}{a} \right) \1_{]0,a,[} + \frac{t}{L -a} \1_{]a,L]}, \quad \hbox{for} \ 0
\leq t \leq T,$$
where $T = \frac{ka(L-a)}{L}$,
and
$$u(t) =     \frac {ka}{L} \1_{]0,L[} , \quad \hbox{for} \  t \geq T.$$

\item[(2)] If $u_0= k\1_{]b,L[}$, then the solution of \eqref{NeumannR} is given by
$$u(t)= \left( k - \frac{t}{L-b} \right) \1_{]b,L[} + \frac{t}{b} \1_{]0,b[}, \quad \hbox{for} \ 0 \leq
t \leq T,$$
where $T = \frac{k(L-b)}{L}$,
and
$$u(t) = k\frac{L -b}{L} \1_{]0,L[} , \quad \hbox{for} \  t \geq T.$$

\item[(3)]  Let $0 < k_1 < k_2$. If $u_0= k_1 \1_{]0,c[} + k_2 \1_{]c,L[}$, then the solution of
    \eqref{NeumannR} is given by
$$u(t)= \left( k_1 + \frac{t}{c} \right) \1_{]0,c[} + \left( k_2 -\frac{t}{L - c} \right) \1_{]c,L[ },
\quad \hbox{for} \ 0 \leq t \leq T,$$
where $T = \frac{(k_2-k_1)c(L-c)}{L}$,
and
$$u(t) = \left(k_1 + \frac{(k_2-k_1)(L-c)}{L} \right) \1_{]0,L[} , \quad \hbox{for} \  t \geq T.$$

\item[(4)] Assume that  $0 < a < b < L$ and also that $L < a+b$ . If $u_0= k\1_{]a,b[}$, then the
    solution of \eqref{NeumannR} is given by
$$u(t)= \frac{t}{a} \1_{]0,a[} + \left(k -  \frac{2}{b-a}t \right)  \1_{]a,b[} + \frac{t}{L-b}
\1_{]b,L[}, \quad \hbox{for} \ 0 \leq t \leq T_1,$$
where $T_1 =  \frac{(b-a)(L-b)}{2L-(a+b)}k$,
$$u(t)= \left(  \frac{T_1}{a}+ \frac{t}{a} \right) \1_{]0,a[} + \left( \left(k -\frac{2}{b-a} T_1\right)
-\frac{t}{L - a} \right) \1_{]a,L[ }, \quad \hbox{for} \ T_1 \leq t \leq T_2,$$
where $$T_2 = \frac{\left(\left(k -\frac{2}{b-a} T_1\right)-\frac{T_1}{a} \right)a(L-a)}{L},$$
and
$$u(t) = \left(  \frac{T_1+T_2}{a} \right) \1_{]0,L[}, \quad \hbox{for} \ t > T_2.$$

\end{itemize}

\end{lemma}

\begin{proof}
(1): Given the initial datum $u_0= \1_{]0,a[}$ we look for a solution of the form $$u(t) = \alpha(t)
\1_{]0,a[} + \beta(t) \1_{]a,L[}$$ on some interval $]0,T[$ defined by the inequality $\alpha(t) >
\beta(t)$ for $t \in ]0,T[$, and $\alpha(0) = k$, $\beta(0) =0$. Then, we shall look for some $\z \in
L^\infty(0,T; L^\infty(]0,L[)$, $\Vert \z(t) \Vert_{L^\infty(]0,L[)} \leq 1$ for almost all $t \in ]0,T[$,
such that
\begin{equation}\label{TNeqq2NE}
   u'(t) = \z(t)^{\prime},\quad  \hbox{in} \ \mathcal{D}^{\prime}(]0,a[),
\end{equation}
\begin{equation}\label{TNeqq2NEE}
   u'(t) = \z(t)^{\prime},\quad  \hbox{in} \ \mathcal{D}^{\prime}(]a,L[),
\end{equation}
\begin{equation}\label{frontNE}
 \z(t)(0) = \z(t)(L)=0,
 \end{equation}
and
\begin{equation}\label{Teqq1NE}
\int_0^L \vert Du(t)\vert  = \int_{\Gamma}   \z(t) Du(t).
\end{equation}

For $0 \leq t \leq T$, we define
$$ \z(t)(x):= \left\{ \begin{array}{ll} - \frac{x}{a}, \quad &\hbox{if} \ 0 \leq x \leq a, \\[10pt]
\frac{x - L}{L - a}, \quad &\hbox{if} \ a \leq x \leq L.\end{array} \right.$$
Integrating equation \eqref{TNeqq2NE}  over $(0,a)$, we obtain
$$\alpha'(t) a  = \int_0^a  \z(t)^{\prime}(x) dx = \z(t)(a) =-1.$$
Thus $\alpha'(t) = - \frac{1}{a}$ and, therefore, $\alpha(t) = k - \frac{t}{a}$. Similarly, we deduce that
$\beta'(t) = \frac{1}{L -a}$, hence $\beta(t) =  \frac{t}{L-a}$. Then, the first $T$ such that $\alpha(T)
= \beta(T)$, is given by $T = \frac{ka(L-a)}{L}$. An easy computation shows that \eqref{Teqq1NE} holds for
all $t \in ]0,T[$. Finally, if we take $\z(t) =0$ for $t >T$, we have that $$u(t) = k\left(1 - \frac{L
-a}{L} \right) \1_{]0,L[}$$
is a solution for $t \geq T.$

The proof of (2) is similar to the proof of (1), taking in this case, for $0 \leq t \leq T$,
$$ \z(t)(x):= \left\{ \begin{array}{ll}  \frac{x}{b}, \quad &\hbox{if} \ 0 \leq x \leq b, \\[10pt]
\frac{L-x}{L - b}, \quad &\hbox{if} \ b \leq x \leq L\end{array} \right.$$

\noindent (3): We look for a solution of the form $$u(t) = \alpha(t)  \1_{]0,c[} + \beta(t) \1_{]c,L[}$$
on some interval $]0,T[$ defined by the inequality $\alpha(t) < \beta(t)$ for $t \in ]0,T[$, and
$\alpha(0) = k_1$, $\beta(0) =k_2$. Working as in the proof of (1), we shall look for some $\z \in
L^\infty(0,T; L^\infty(]0,L[)$, with $\Vert \z(t) \Vert_{L^\infty(]0,L[)} \leq 1$ for almost all $t \in
]0,T[$ and  $\z(t)(0) = \z(t)(L)=0$, satisfying
$$\alpha'(t) c  = \int_0^c  \z(t)^{\prime}(x) dx = \z(t)(c)$$
and
$$\beta'(t) (L-c)  = \int_c^L  \z(t)^{\prime}(x) dx = -\z(t)(c).$$
Then,
$$\alpha(t) = k_1 + \frac{\z(t)(c)}{c}, \quad \beta(t) = k_2 - \frac{\z(t)(c)}{L-c}.$$
Hence, taking, for $0 \leq t \leq T$,
$$ \z(t)(x):= \left\{ \begin{array}{ll}  \frac{x}{c}, \quad &\hbox{if} \ 0 \leq x \leq c, \\[10pt]
\frac{L-x}{L - c}, \quad &\hbox{if} \ c \leq x \leq L,\end{array} \right.$$
it is easy to see that   $$u'(t) = \z(t)^{\prime},\quad  \hbox{in} \ \mathcal{D}^{\prime}(]0,L[),$$
and
$$\int_0^L \vert Du(t)\vert  = \int_{\Gamma}   \z(t) Du(t).$$
Therefore, for $0 < t \leq T = \frac{(k_2-k_1)c(L-c)}{L}$, the solution is given by
$$u(t)= \left( k_1 + \frac{t}{c} \right) \1_{]0,c[} + \left( k_2 -\frac{t}{L - c} \right) \1_{]c,L[ }.$$
Moreover,
$$u(t) = \left(k_1 + \frac{(k_2-k_1)(L-c)}{L} \right) \1_{]0,L[} , \quad \hbox{for} \  t \geq T.$$

\noindent (4): In this case we look for a solution of the form
$$u(t) = \alpha(t)  \1_{]0,a[} + \beta(t) \1_{]a,b[} + \gamma(t) \1_{]b,L[} $$ on some interval $(0,T)$
defined by the inequality $\alpha(t) < \beta(t)$, $\gamma(t) < \beta(t)$  for $t \in ]0,T[$, and
$\alpha(0) = \gamma(t) =0$, $\beta(0) =k$.  Working as in the proof of (1),  we need to find a vector
field $\z \in L^\infty(0,T; L^\infty(]0,L[)$, $\Vert \z(t) \Vert_{L^\infty(]0,L[)} \leq 1$, for almost all
$t \in ]0,T[$, satisfying
$$\alpha(t)= \frac{\z(t)(a)}{a}, \quad \beta(t) = k + \left( \frac{\z(t)(b)- \z(t)(a)}{b-a} \right)t,
\quad \gamma(t) = - \frac{\z(t)(b)}{L-b}.$$
Now, if we take, for $0 \leq t \leq T$,
$$ \z(t)(x):= \left\{ \begin{array}{lll}  \frac{x}{a}, \quad &\hbox{if} \ 0 \leq x \leq a, \\[10pt] -2
\frac{x-a}{b-a} +1, \quad &\hbox{if} \ a \leq x \leq b, \\[10pt]  \frac{x - L}{L-b}, \quad &\hbox{if} \ b
\leq x \leq L.\end{array} \right.$$
Hence,
$$\alpha(t)= \frac{t}{a}, \quad \beta(t) = k + \left( \frac{-2}{b-a} \right)t, \quad \gamma(t) =
\frac{t}{L-b}.$$
Since we are assuming that $L-b < a$, we have $\alpha(t) < \gamma(t)$. Then, for
$$T_1:= \frac{(b-a)(L-b)}{2L-(a+b)}k,$$
we have $\beta(T_1) = \gamma(T_1)$. Hence, for $0 < t \leq T_1$, if
$$u(t) = \frac{t}{a} \1_{]0,a[} + \left(k -  \frac{2}{b-a}t \right)  \1_{]a,b[} + \frac{t}{L-b}
\1_{]b,L[},$$
       it is easy to see that   $$u'(t) = \z(t)^{\prime}\quad  \hbox{in} \ \mathcal{D}^{\prime}(]0,L[),$$
and
$$\int_0^L \vert Du(t)\vert  = \int_{\Gamma}   \z(t) Du(t).$$
Therefore, for $0 < t \leq T_1 $, $u(t)$ is the solution. Now,
$$u(T_1) =  \frac{T_1}{a} \1_{]0,a[} + \left(k -\frac{2}{b-a} T_1\right)  \1_{]a,L[}.$$
Then, applying (3), we have
$$u(t)= \left(  \frac{T_1}{a}+ \frac{t}{a} \right) \1_{]0,a[} + \left( \left(k -\frac{2}{b-a} T_1\right)
-\frac{t}{L - a} \right) \1_{]a,L[ }, \quad \hbox{for} \ T_1 \leq t \leq T_2,$$
where
$$T_2 = \frac{a(L-a)}{L}\left(k - T_1\frac{a+b}{a(b-a)}\right).$$
Finally, for $t > T_2$, the solution in given by
$$u(t) = \left(  \frac{T_1+T_2}{a} \right) \1_{]0,L[}.$$
\end{proof}

\begin{remark}{\rm Let us  point out that it is obtained in \cite{BF} that for the initial data $u_0 = k_1
\1_{(a,b)} + k_2 \1_{(b,L)}$ with $0 < k_1 < k_2$,  the solution of \eqref{NeumannR} is given by
$$u(t) = \frac{t}{a} \1_{]0,a[} + k_1 \1_{]a,b[} + \left(k_2 - \frac{t}{L-b} \right)\1_{]b,L[}, $$
for $0 \leq t \leq T_1$, with
$$T_1 = \min \{ ak_1, (k_2 - k_1)(L - b) \}.$$

}
\end{remark}

 We are now going to find an explicit solution in the case of a simpler metric graph in order to see the
 difference in behaviour with the case of the total variation flow in an interval with Neumann boundary
 conditions that we have considered in the above result.

\begin{example}{\rm Consider the metric graph $\Gamma$ with three vertices and two  edges, that is
$V(\Gamma) = \{\v_1, \v_2, \v_3 \}$ and $E(\Gamma) = \{ \e_1:=[\v_1, \v_2], \e_2:=[\v_2, \v_3], \}$. Let
$0 < a < \ell_{\e_2}$ and assume that $\ell_{\e_1} > \ell_{\e_2} - a$. We are going to find the solution
of the total variation flow for the initial datum $u_0:= k \1_D$, with $k >0$ and $D:= (\v_2,
c_{\e_2}^{-1}(a))$.

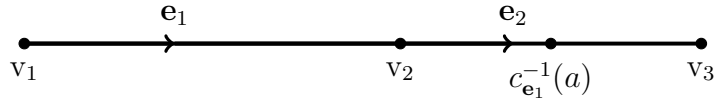
\begin{center}
\begin{tikzpicture}

\node[below] at (-4,-0.1) {$\v_1$};
\node[below]  at (1,-0.1) {$\v_2$};
\node[below]  at (5,-0.1) {$\v_3$};

\draw[line width=1.5pt,->]  (-4,0) -- (-2,0);
\draw[line width=1.5pt,->]   (-2,0) -- (2.5,0);
\draw[line width=1.5pt]    (-4,0) -- (5,0);

\node[above] at (-2,0.1) {$\e_1$};
\node[above]  at (2.5,0.1) {$\e_2$};
\node[below] at (3,-0.1) {$c^{-1}_{\e_1}(a)$};

\draw[fill=black] (-4,0) circle (2pt);
\draw[fill=black] (1,0) circle (2pt);
\draw[fill=black] (5,0) circle (2pt);
\draw[fill=black] (3,0) circle (2pt);
\end{tikzpicture}
\end{center}

 We look for solutions of the form:
 $$[u(t)]_{\e_1} = \alpha_1(t) \1_{]0,\ell_{\e_1}[} \quad \alpha_1(0) =0,$$
 $$[u(t)]_{\e_2} = \alpha_2(t) \1_{]0,a[} + \alpha_3(t) \1_{]a, \ell_{\e_2}[}, \quad \alpha_2(0) =k, \
 \alpha_3(0) =0,$$
 for all $0 < t \leq T$ such that $$\alpha_1(t) \leq \alpha_2(t), \quad \alpha_2(t) \leq \alpha_3(t).$$
 Then, we need to find $\z(t) \in X_K(\Gamma)$, with $\Vert \z(t) \Vert_\infty \leq 1$, satisfying:
 \begin{equation}\label{E1Imp}
 [u(t)]_{\e_i}^{\prime} = [\z(t)]_{\e_i}^{\prime},\quad i=1,2, \ \hbox{that is}
  \end{equation}
  $$\alpha^{\prime}_1(t) \1_{]0, \ell_{\e_1}[} = [\z(t)]_{\e_1}^{\prime},\quad \alpha^{\prime}_2(t)
  \1_{]0,a[} + \alpha_3^{\prime}(t) \1_{]a, \ell_{\e_2}[} = [\z(t)]_{\e_2}^{\prime}.$$
 \begin{equation}\label{E2Imp} \begin{array}{ll}
TV_\Gamma (u(t)) = \displaystyle\int_{\Gamma}   \z(t) Du(t) \\[10pt] -  \frac{1}{2}  \left([\z(t)]_{\e_1}
(\v_2)([u(t)]_{\e_1}(\v_2) - [u(t)]_{\e_2}(\v_2)) + [\z(t)]_{\e_2} (\v_2)([u(t)]_{\e_2}(\v_2) -
[u(t)]_{\e_1}(\v_2))  \right). \end{array}
 \end{equation}
By \eqref{Form1Igual}, we can write \eqref{E2Imp} as
$$ \begin{array}{ll}
\vert Du(t) \vert(\Gamma) + JV_\Gamma (u(t)) = \displaystyle\int_{\Gamma}   \z(t) Du(t) \\[10pt] =-
\frac{1}{2}  \left([\z(t)]_{\e_1} (\v_2)([u(t)]_{\e_1}(\v_2) - [u(t)]_{\e_2}(\v_2)) + [\z(t)]_{\e_2}
(\v_2)([u(t)]_{\e_2}(\v_2) - [u(t)]_{\e_1}(\v_2))  \right). \end{array}$$
 Now,
 $$Du(t) = (\alpha_3(t)- \alpha_2(t)) \delta_a.$$
 Hence,
 $$\vert Du(t) \vert(\Gamma) = (\alpha_2(t)- \alpha_3(t)),$$
 and
 $$ \int_{\Gamma}   \z(t) Du(t) =  (\alpha_3(t)- \alpha_2(t))[\z]_{\e_2}(a).$$
 Thus, if we assume that  $[\z]_{\e_2}(a) = -1$, and having in mind that $[\z(t)]_{\e_1} (\v_2) = -
 [\z(t)]_{\e_2} (\v_2)$,   we have that  we can rewrite \eqref{E2Imp} as
 $$ JV_\Gamma (u(t)) = -  [\z(t)]_{\e_1} (\v_2)([u(t)]_{\e_1}(\v_2) - [u(t)]_{\e_2}(\v_2)) = -
 [\z(t)]_{\e_1}(\v_2) (\alpha_1(t)- \alpha_2(t)).$$
 Now,
 $$JV_\Gamma (u(t)) =  \vert [u(t)]_{\e_1}(\v_2) - [u(t)]_{\e_2}(\v_2)\vert =\alpha_2(t)- \alpha_1(t),$$
 and then, \eqref{E2Imp}  is equivalent to
$$
 \alpha_2(t)- \alpha_1(t)  =   -  [\z(t)]_{\e_1}(\v_2) (\alpha_1(t)- \alpha_2(t)).
$$
Therefore, if $[\z(t)]_{\e_1}(\v_2) =[\z(t)]_{\e_1}(\ell_{\e_1}) =1$, we have that \eqref{E2Imp}  holds.

 We define
   $$[\z(t)]_{\e_1}(x):= \frac{x}{\ell_{\e_1}}, \quad \hbox{if} \ \ 0 \leq x \leq \ell_{\e_1}, $$
   and
    $$[\z(t)]_{\e_2}(x):= \left\{    \begin{array}{ll} -\frac{2x}{a} +1, \quad &\hbox{if} \ \ 0 \leq x
    \leq a, \\[10pt] \frac{  x - \ell_{\e_2}}{\ell_{\e_2}- a}, \quad &\hbox{if} \ \ a \leq x \leq
    \ell_{\e_2}. \end{array} \right.$$
Note that
$$[\z(t)]_{\e_1}(\v_2) + [\z(t)]_{\e_2}(\v_2) = [\z(t)]_{\e_1}(\ell_{\e_1}) - [\z(t)]_{\e_2}(0)= 0,$$
thus $\z(t) \in X_K(\Gamma)$.

 On the other hand, integrating in \eqref{E1Imp}, we get
  $$\alpha^{\prime}_1(t)\ell_{\e_1} = \int_0^{\ell_{\e_1}} [\z(t)]_{\e_1}^{\prime} dx =
  [\z(t)]_{\e_1}(\ell_{\e_1}) \ \Rightarrow \ \alpha_1(t) =
  \frac{[\z(t)]_{\e_1}(\ell_{\e_1})}{\ell_{\e_1}}t = \frac{1}{\ell_{\e_1}}t,$$
  $$\alpha^{\prime}_2(t)a = \int_0^a [\z(t)]_{\e_2}^{\prime} dx = [\z(t)]_{\e_2}(a) - [\z(t)]_{\e_2}(0)\
  \Rightarrow \ \alpha_2(t) = k + \frac{[\z(t)]_{\e_2}(a)- [\z(t)]_{\e_2}(0)}{a} = k - \frac{2}{a} t,$$
   $$\alpha^{\prime}_3(t)(\ell_{\e_2}-a)  = \int_{a}^{\ell_{\e_2}} [\z(t)]_{\e_2}^{\prime} dx = -
   [\z(t)]_{\e_2}(a)\ \Rightarrow \ \alpha_3(t)= - \frac{[\z(t)]_{\e_2}(a)}{\ell_{\e_2}-a}t =
   \frac{1}{\ell_{\e_2}-a}t.$$

    Consequently, since $\ell_{\e_1} > \ell_{\e_2} - a$, the solution is given by
    $$[u(t)]_{\e_1} =  \frac{t}{\ell_{\e_1}} \1_{]0,\ell_{\e_1}[}, \quad \hbox{for} \ 0 \leq t \leq
    T_1,$$
       and
         $$[u(t)]_{\e_2} =  \displaystyle\left( k - \frac{2t}{a} \right) \1_{]0,a[} + \frac{t}{\ell_{\e_2}
         -a} \1_{]a,\ell_{\e_2}[}, \quad \hbox{for} \ 0 \leq t \leq T_1,$$
          where $$T_1 =  \frac{ka(\ell_{\e_2}-a)}{2\ell_{\e_2}-a}.$$
          We have
          $$[u(T_1)]_{\e_2} =  \displaystyle\left( k - \frac{2T_1}{a} \right) \1_{]0,\ell_{\e_2}[} =
          \displaystyle\left( k - \frac{2k(\ell_{\e_2}-a)}{2\ell_{\e_2}-a} \right) \1_{]0,\ell_{\e_2}[} =
          k \frac{a}{2\ell_{\e_2}-a} \1_{]0,\ell_{\e_2}[}.$$

          Now, for $t > T_1$, we look for a solution of the form
           $$[u(t)]_{\e_1} = \gamma_1(t) \1_{]0,\ell_{\e_1}[} \quad \gamma_1(T_1) =\alpha_1(T_1),$$
 $$[u(t)]_{\e_2} = \gamma_2(t) \1_{]0,\ell_{\e_2}[} \quad \gamma_2(T_1) = \alpha_2(T_1),$$
 for all $T_1 < t \leq T _2$ such that $$\gamma_1(t) \leq \gamma_2(t).$$

  Then, we need to find $\z(t) \in X_K(\Gamma)$, with $\Vert \z(t) \Vert_\infty \leq 1$, satisfying:
 \begin{equation}\label{E1ImpNN}
 [u(t)]_{\e_i}^{\prime} = [\z(t)]_{\e_i}^{\prime},\quad i=1,2, \ \hbox{that is},
  \end{equation}
  $$\gamma^{\prime}_1(t) \1_{]0, \ell_{\e_1}[} = [\z(t)]_{\e_1}^{\prime},\quad \gamma^{\prime}_2(t)
  \1_{]0,\ell_{\e_2}[} = [\z(t)]_{\e_2}^{\prime}.$$
 \begin{equation}\label{E2ImpNN} \begin{array}{ll}
TV_\Gamma (u(t)) = \displaystyle\int_{\Gamma}   \z(t) Du(t) \\[10pt] -  \frac{1}{2}  \left([\z(t)]_{\e_1}
(\v_2)([u(t)]_{\e_1}(\v_2) - [u(t)]_{\e_2}(\v_2)) + [\z(t)]_{\e_2} (\v_2)([u(t)]_{\e_2}(\v_2) -
[u(t)]_{\e_1}(\v_2))  \right). \end{array}
 \end{equation}
 By \eqref{Form1Igual}, we can write \eqref{E2ImpNN} as
$$ \begin{array}{ll}
\vert Du(t) \vert(\Gamma) + JV_\Gamma (u(t)) = \displaystyle\int_{\Gamma}   \z(t) Du(t) \\[10pt] = -
\frac{1}{2}  \left([\z(t)]_{\e_1} (\v_2)([u(t)]_{\e_1}(\v_2) - [u(t)]_{\e_2}(\v_2)) + [\z(t)]_{\e_2}
(\v_2)([u(t)]_{\e_2}(\v_2) - [u(t)]_{\e_1}(\v_2))  \right), \end{array}$$
which, having in mind that $[\z(t)]_{\e_1} (\v_2) + [\z(t)]_{\e_2} (\v_2) =0$, is equivalent to
$$\gamma_2(t) - \gamma_1(t) = \vert [u(t)]_{\e_1}(\v_2) - [u(t)]_{\e_2}(\v_2) \vert = -[\z(t)]_{\e_1}
(\v_2)(\gamma_1(t) - \gamma_2(t)).$$
Then, if $[\z(t)]_{\e_1} (\v_2) = 1$, we have that \eqref{E2ImpNN} holds.

 We define
   $$[\z(t)]_{\e_1}(x):= \frac{x}{\ell_{\e_1}}, \quad \hbox{if} \ \ 0 \leq x \leq \ell_{\e_1}, $$
   and
    $$[\z(t)]_{\e_2}(x):=  \frac{ \ell_{\e_2}- x }{\ell_{\e_2}}, \quad \hbox{if} \ \ 0 \leq x \leq
    \ell_{\e_2}. $$

    Now, integrating in \eqref{E1ImpNN}, for $T_1 < t \leq T_2$, we get
  $$\gamma^{\prime}_1(t)\ell_{\e_1} = \int_0^{\ell_{\e_1}} [\z(t)]_{\e_1}^{\prime} dx =
  [\z(t)]_{\e_1}(\ell_{\e_1}) =1 \ \Rightarrow \ \gamma_1(t) = \alpha_1(T_1) + \frac{1}{\ell_{\e_1}}t,$$
  $$\gamma^{\prime}_2(t)\ell_{\e_2} = \int_0^{\ell_{\e_2}} [\z(t)]_{\e_2}^{\prime} dx =- [\z(t)]_{\e_2}(0)
  = -1 \ \Rightarrow \ \gamma_2(t) = \alpha_2(T_1) - \frac{1}{\ell_{\e_2}} t,$$
  where $T_2$ is given by
  $$\alpha_1(T_1) + \frac{1}{\ell_{\e_1}}T_2 = \alpha_2(T_1) - \frac{1}{\ell_{\e_2}} T_2,$$
  that is,
  $$T_2 = \ell_{\e_1} \ell_{\e_2} \frac{\alpha_2(T_1) - \alpha_1(T_1)}{\ell_{\e_1}+ \ell_{\e_2}}.$$

  Consequently,  the solution $u(t)$ of the Cauchy problem \eqref{CP1} for the initial datum $u_0:= k
  \1_D$ is given by
     $$[u(t)]_{\e_1} = \left\{\begin{array}{ll}  \displaystyle\frac{t}{\ell_{\e_1}} \1_{]0,\ell_{\e_1}[},
     \quad \hbox{for} \ 0 \leq t \leq T_1, \\ \\  \displaystyle \frac{1}{\ell_{\e_1}}
     \left(\frac{ka(\ell_{\e_2}-a)}{2\ell_{\e_2}-a} + t \right)  \1_{]0,\ell_{\e_1}[} , \quad \hbox{for} \
     T_1 \leq t \leq T_2 \end{array} \right.$$
        and
         $$[u(t)]_{\e_2} = \left\{\begin{array}{ll} \displaystyle\left( k - \frac{2t}{a} \right)
         \1_{(0,a)} + \frac{t}{\ell_{\e_2} -a} \1_{]a,\ell_{\e_2}[}, \quad \hbox{for} \ 0 \leq t \leq T_1
         \\ \\  \left( k \frac{a}{2\ell_{\e_2}-a}  - \frac{t}{\ell_{\e_2}} \right)\1_{]0,\ell_{\e_2}[},
         \quad \hbox{for} \  T_1 \leq  t \leq T_2 \end{array} \right.$$
        where
        $$T_1 =  \frac{ka(\ell_{\e_2}-a)}{2\ell_{\e_2}-a}, \quad \hbox{and} \quad T_2 = \ell_{\e_1}
        \ell_{\e_2} \frac{\alpha_2(T_1) - \alpha_1(T_1)}{\ell_{\e_1}+ \ell_{\e_2}}.$$
Moreover, for $t \geq T_2$,

$$u(t) = \frac{T_1}{\ell_{\e_1}} =  \ell_{\e_2} \frac{\alpha_2(T_1) - \alpha_1(T_1)}{\ell_{\e_1}+
\ell_{\e_2}}.$$

}
\end{example}

\begin{remark}{\rm
Let us point out that in the above example, we see that  the solution  does not coincide with the solution
of the Neumann problem in each edge.  However, this happens if we consider that the total variation of a
function $u$ is given by $\vert Du  \vert (\Gamma)$, in which case it does not take into account the
structure of the metric graph.
$\blacksquare$}

\end{remark}

\begin{example}{\rm Consider the metric graph $\Gamma$ of the example \ref{example2}

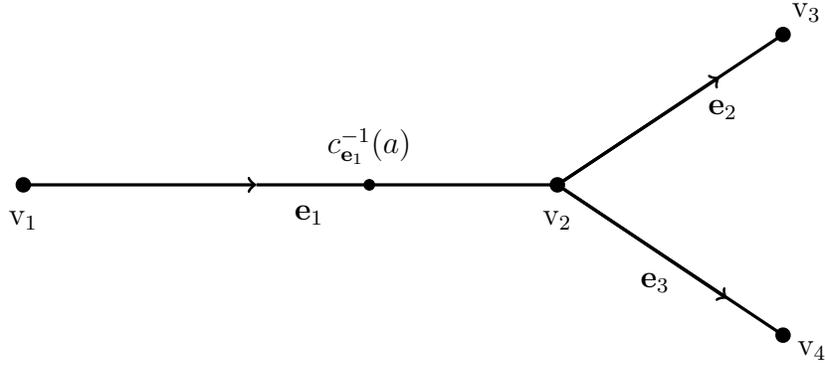
\begin{center}
\begin{tikzpicture}
 \tikzstyle{gordito2} = [line width=3]

\node (v1) at (-5.1,-1.2) {};
\node (v5) at (2,-1.2) {};

\node[below] at (v1) {$\v_1$};
\node[above] at (5.3093,1.0278) {$\v_3$};
\node[below] at (v5) {$\v_2$};
\node[above] at (5.3843,-3.4747) {$\v_4$};

\draw[gordito2]  (-5.1,-1) node (v2) {} circle (0.05);
\draw[gordito2]  (5,1) node (v3) {} circle (0.05);
\draw[gordito2]  (2,-1) node (v4) {} circle (0.05);
\draw[gordito2]  (5,-3) node (v4) {} circle (0.05);

\draw[->,line width=1.2]  (-5.1,-1)--(-2,-1);
\draw[line width=1.2]  (-2,-1) --(2,-1);
\draw[line width=1.2] (2,-1)--(5,-3);
\draw[->,line width=1.2] (2,-1)--(4.253,-2.5103);

\draw[line width=1.2](2,-1)--(5,1);

\draw[->,line width=1.2](2,-1)--(4.1524,0.4299);

\node[below] at (-0.5,-0.1) {$c^{-1}_{\e_1}(a)$};
\node at (-1.3,-1.4) {$\e_1$};
\node at (4.2,0) {$\e_2$};
\node at (3.3,-2.3) {$\e_3$};
\draw[fill=black] (-0.5,-1) circle (2pt);
\end{tikzpicture}
\end{center}

Assume that $\ell := \ell_{\e_2} = \ell_{\e_3}$ and let $0 < a < \ell_{\e_1}$ such that  $a < 2 \ell$. We
are going to find the solution of the total variation flow for the initial datum $u_0:= k \1_D$, with $k
>0$ and $D:= (c_{\e_1}^{-1}(a), \v_2)$.

 We look for solutions of the form:
 $$[u(t)]_{\e_1} = \alpha_1(t) \1_{]0,a[} + \alpha_2(t)\1_{]a,\ell_{\e_1}[} \quad \alpha_1(0) =0,
 \alpha_2(0) =k,$$
 $$[u(t)]_{\e_2} =[u(t)]_{\e_3} =  \beta(t) \1_{(0,\ell)}, \quad \beta(0) =0,$$
 for all $0 < t \leq T_1$ such that $$\alpha_1(t) \leq \alpha_2(t), \quad \beta(t) \leq \alpha_2(t).$$
 Then, we need to find $\z(t) \in X_K(\Gamma)$, with $\Vert \z(t) \Vert_\infty \leq 1$, satisfying:
 \begin{equation}\label{NE1Imp}
 [u(t)]_{\e_1}^{\prime} = [\z(t)]_{\e_1}^{\prime},\quad [u(t)]_{\e_i}^{\prime} = [\z(t)]_{\e_i}^{\prime}
 i=2,3 \ \hbox{that is}
  \end{equation}
  $$\alpha^{\prime}_1(t) \1_{]0, a[} + \alpha^{\prime}_2(t)  \1_{]a, \ell_{\e_1}[} =
  [\z(t)]_{\e_1}^{\prime},\quad \beta^{\prime}(t) \1_{(0,\ell)} = [\z(t)]_{\e_i}^{\prime}, \  i=2,3.$$
 \begin{equation}\label{NE2Imp}
TV_\Gamma (u(t)) = \displaystyle\int_{\Gamma}   \z(t) Du(t)  -  \sum_{i=1}^3  [\z(t)]_{\e_i}
(\v_2)([u(t)]_{\e_i}(\v_2).
 \end{equation}
Now
$$Du(t) = (\alpha_2(t) - \alpha_1(t)) \delta_a,$$
hence
$$\int_{\Gamma}   \z(t) Du(t) = (\alpha_2(t) - \alpha_1(t))  [\z(t)]_{\e_1}(a).$$
 Since $\z(t) \in X_K(\Gamma)$,  $[\z(t)]_{\e_1} (\v_2) = -[\z(t)]_{\e_2} (\v_2) - [\z(t)]_{\e_3} (\v_2)$,
 thus
$$\sum_{i=1}^3  [\z(t)]_{\e_i} (\v_2)([u(t)]_{\e_i}(\v_2) =  [\z(t)]_{\e_1} (\v_2)(\alpha_2(t) -
\beta(t)).$$
Therefore, we can write \eqref{NE2Imp}  as
 \begin{equation}\label{NE3Imp}
TV_\Gamma (u(t)) =  (\alpha_2(t) - \alpha_1(t))  [\z(t)]_{\e_1}(a) -  [\z(t)]_{\e_1} (\v_2)(\alpha_2(t) -
\beta(t)).
 \end{equation}
 On the other hand,
 $$TV_\Gamma(u(t)) =   \sup \left\{ \displaystyle \left\vert \int_{\Gamma} u(t)(x) \w^{\prime}(x) dx
 \right\vert \ : \ \w \in X_K(\Gamma), \ \Vert \w \Vert_{L^\infty(\Gamma)} \leq 1 \right\} $$ $$= \sup
 \left\{ \displaystyle \left\vert  \sum_{i=1}^3  \int_0^{\ell_{\e_i}} [u(t)]_{\e_i}(x)
 [\w]_{\e_i}^{\prime}(x) dx \right\vert \ : \ \w \in X_K(\Gamma), \ \Vert \w \Vert_{L^\infty(\Gamma)} \leq
 1 \right\}.$$
 Now,
 $$ \sum_{i=1}^3  \int_0^{\ell_{\e_i}} [u(t)]_{\e_i}(x) [\w]_{\e_i}^{\prime}(x) dx =  \alpha_1(t) \int_0^a
 [\w]_{\e_1}^{\prime}(x) dx +\alpha_2(t) \int_a^{\ell_{\e_1}} [\w]_{\e_1}^{\prime}(x) dx $$
$$+\beta(t) \int_0^{\ell_{\e_2}} [\w]_{\e_2}^{\prime}(x) dx  + \beta(t) \int_0^{\ell_{\e_3}}
[\w]_{\e_3}^{\prime}(x) dx = (\alpha_1(t)-\alpha_2(t)) [\w]_{\e_1}(a)$$ $$+\alpha_2(t)) [\w]_{\e_1}(\v_2)+
\beta([\w]_{\e_2}(\v_2) + [\w]_{\e_3}(\v_2))$$ $$= (\alpha_1(t)-\alpha_2(t)) [\w]_{\e_1}(a) + (\alpha_2(t)
- \beta(t)) [\w]_{\e_1}(\v_2).$$
Thus
$$TV_\Gamma(u(t))$$ $$ =   \sup \left\{ \vert (\alpha_1(t)-\alpha_2(t)) [\w]_{\e_1}(a) + (\alpha_2(t) -
\beta(t)) [\w]_{\e_1}(\v_2) \vert \ : \ \w \in X_K(\Gamma), \ \Vert \w \Vert_{L^\infty(\Gamma)} \leq 1
\right\}$$ $$= (\alpha_2(t)-\alpha_1(t)) + (\alpha_2(t) - \beta(t)). $$
Then, if $[\z(t)]_{\e_1}(a)  =1$ and $[\z(t)]_{\e_1}(\v_2) = -1$,  \eqref{NE3Imp} holds.

We define

$$ [\z(t)]_{\e_1}(x):= \left\{ \begin{array}{ll}  \displaystyle\frac{x}{a}, \quad &\hbox{if} \ 0 \leq x
\leq a, \\[10pt] \displaystyle\frac{\ell_{\e_1}+ a - 2x}{\ell_{\e_1}- a }  \quad &\hbox{if} \ a \leq x
\leq \ell_{\e_1}.\end{array} \right.$$
Now, integrating in \eqref{NE1Imp}, we get
 $$a \alpha^{\prime}_1(t) = \int_0^a [\z(t)]_{\e_1}^{\prime}(x) dx = [\z(t)]_{\e_1}(a) =1 \ \Rightarrow \
 \alpha_1(t) = \frac{t}{a},$$
  $$\alpha^{\prime}_2(t)(\ell_{\e_1}-a) = \int_a^{\ell_{\e_i}} [\z(t)]_{\e_1}^{\prime}(x) dx =
  [\z(t)]_{\e_1}(\ell_{\e_1}) - [\z(t)]_{\e_1}(a) = - 2 \ \Rightarrow \  \alpha_2(t) = \left( k -
  \frac{2t}{\ell_{\e_1}-a} \right),$$

  $$\hbox{for $i =2,3$}, \ \beta^{\prime}(t)\ell_{\e_i} = \int_0^{\ell_{\e_i}} [\z(t)]_{\e_i}^{\prime}(x)
  dx = - [\z(t)]_{\e_i}(0) = [\z(t)]_{\e_i}(\v_2) = \frac12  \ \Rightarrow \  \beta(t) = \frac{t}{2
  \ell}.$$
   Consequently, the solution is given by
    $$[u(t)]_{\e_1} =  \frac{t}{a} \1_{]0,a[} +  \left( k - \frac{2t}{\ell_{\e_1}-a} \right)
    \1_{]a,\ell_{\e_1}[} \quad \hbox{for} \ 0 \leq t \leq T_1,$$
       and
         $$[u(t)]_{\e_2} = [u(t)]_{\e_3} =  \frac{t}{2 \ell}\1_{]0,\ell[}, \quad \hbox{for} \ 0 \leq t
         \leq T_1,$$
          where $$T_1 = \frac{k a (\ell_{\e_1}-a)}{\ell_{\e_1}+a},$$
          since we are  assuming that $a < 2 \ell$.

           Now, for $t > T_1$, we look for a solution of the form
           $$[u(t)]_{\e_1} =  \gamma_1(t)\1_{]0,\ell_{\e_1}[},$$
            $$[u(t)]_{\e_i} =  \gamma_2(t)\1_{]0,\ell_{\e_i}[} , \ i=2,3,$$
             with
            $$\gamma_1(T_1) = \frac{T_1}{a} = \left( k - \frac{2T_1}{\ell_{\e_1}-a} \right), \quad
            \gamma_2(T_1) =  \gamma_3(T_1) = \frac{T_1}{2 \ell}, \ i=2,3,$$
            such that $$\gamma_1(t) \geq \gamma_i(t),   \quad i=2,3., \quad \hbox{for} \ T_1 \leq t \leq
            T_2.$$
             Then, we need to find $\z(t) \in X_K(\Gamma)$, with $\Vert \z(t) \Vert_\infty \leq 1$,
             satisfying:
 \begin{equation}\label{NNE1ImpN}
 [u(t)]_{\e_1}^{\prime} = [\z(t)]_{\e_1}^{\prime},\quad [u(t)]_{\e_i}^{\prime} = [\z(t)]_{\e_i}^{\prime}
 i=2,3, \ \ \hbox{that is}
  \end{equation}
  $$\gamma^{\prime}_1(t) \1_{(0,  \ell_{\e_1})}  = [\z(t)]_{\e_1}^{\prime},\quad \gamma_i^{\prime}(t)
  \1_{(0,\ell)} = [\z(t)]_{\e_i}^{\prime}, \  i=2,3$$
  and
   \begin{equation}\label{NNE2ImpN}
TV_\Gamma (u(t)) = \displaystyle\int_{\Gamma}   \z(t) Du(t)  -  \sum_{i=1}^3  [\z(t)]_{\e_i}
(\v_2)[u(t)]_{\e_i}(\v_2).
 \end{equation}
  Now $Du(t) = 0,$
hence
$$\int_{\Gamma}   \z(t) Du(t) = 0.$$
Since $\z(t) \in X_K(\Gamma)$, we have
$$ -  \sum_{i=1}^3  [\z(t)]_{\e_i} (\v_2)[u(t)]_{\e_i}(\v_2) = - [\z(t)]_{\e_i} (\v_2)(\gamma_1(t) -
\gamma_2(t)).$$
On the other hand,
 $$TV_\Gamma(u(t)) = \sup \left\{ \displaystyle \left\vert  \sum_{i=1}^3  \int_0^{\ell_{\e_i}}
 [u(t)]_{\e_i}(x) [\w]_{\e_i}^{\prime}(x) dx \right\vert \ : \ \w \in X_K(\Gamma), \ \Vert \w
 \Vert_{L^\infty(\Gamma)} \leq 1 \right\}.$$
 Now,
 $$ \sum_{i=1}^3  \int_0^{\ell_{\e_i}} [u(t)]_{\e_i}(x) [\w]_{\e_i}^{\prime}(x) dx $$ $$=  \gamma_1(t)
 \int_0^{\ell_{\e_1}} [\w]_{\e_1}^{\prime}(x) dx
+\gamma_2(t) \int_0^{\ell_{\e_2}} [\w]_{\e_2}^{\prime}(x) dx  + \gamma_3(t) \int_0^{\ell_{\e_3}}
[\w]_{\e_3}^{\prime}(x) dx $$ $$=  \gamma_1(t)[\w]_{\e_1}(\ell_{\e_1}) - \gamma_2(t) [\w]_{\e_2}(0) -
\gamma_3(t) [\w]_{\e_3}(0 $$ $$=  \gamma_1(t)[\w]_{\e_1}(\v_2) + \gamma_2(t) [\w]_{\e_2}(\v_2) +
\gamma_3(t) [\w]_{\e_3}(\v_2)=  (\gamma_1(t) -  \gamma_2(t))[\w]_{\e_1}(\v_2).$$
Hence,
$$TV_\Gamma(u(t)) = (\gamma_1(t) -  \gamma_2(t)).$$
Therefore, \eqref{NNE2ImpN} holds , if $[\z(t)]_{\e_1}(\v_2) = -1$.

Now, integrating \eqref{NNE1ImpN}, for $T_1 \leq t \leq T_2$, we have
$$\gamma_1(t)^{\prime} \ell_{\e_1} = \int_0^{\ell_{\e_1}} [\z(t)]^{\prime}_{\e_1} (x)dx =
[\z(t)]_{\e_1}(\v_2) = -1 \ \Rightarrow \ \gamma_1(t) = \frac{T_1}{a} - \frac{t}{\ell_{\e_1}},$$

$$\hbox{for $i =2,3$}, \ \gamma_i(t)^{\prime} \ell_{\e_i} = \int_0^{\ell_{\e_i}} [\z(t)]^{\prime}_{\e_i}
(x)dx = [\z(t)]_{\e_i}(\v_2) = \frac12 \ \Rightarrow \ \gamma_i(t) = \frac{T_1}{2 \ell} +
\frac{t}{2\ell}.$$
Consequently, the solution is given by
    $$[u(t)]_{\e_1} = \frac{T_1}{a} - \frac{t}{\ell_{\e_1}} \1_{]0,\ell_{\e_1}[} \quad \hbox{for} \ \leq
    T_1 \leq t \leq T_2,$$
       and
         $$[u(t)]_{\e_2} = [u(t)]_{\e_3} =  \frac{T_1}{2 \ell} + \frac{t}{2\ell}, \quad \hbox{for} \  \leq
         T_1 \leq t \leq T_2,$$

where
$$T_2 = T_1 \frac{(2\ell-a)\ell_{\e_1}}{a(\ell_{\e_1} +2\ell)}$$

For $t \geq T_2$, we have
$$u(t) = \frac{T_1}{a} - \frac{T_2}{\ell_{\e_1}} = T_1 \left(\frac{1}{a} - \frac{(2\ell-a)}{a(\ell_{\e_1}
+2\ell)}  \right) = T_1 \frac{\ell_{\e_1}+a}{(\ell_{\e_1} +2\ell)} = k \frac{\ell_{\e_1}-a}{(\ell_{\e_1}
+2\ell)} = \overline{u_0}.$$

          }

\end{example}

\noindent {\bf Acknowledgment.}  The author is grateful to  Wojciech G\'{o}rny, Delio Mugnolo  and Juli\'{a}n  Toledo for
stimulating discussions on this paper. The author have been partially supported  by the Spanish MCIU and
FEDER, project PGC2018-094775-B-100.

 \end{document}